\documentclass[11pt, intlimits]{amsart}
\usepackage{amsmath,amssymb}
\usepackage{graphicx}
\usepackage{verbatim}

\usepackage{amsmath,amssymb}
\usepackage{graphicx}
\usepackage{amsbsy}
\usepackage{amsfonts}
\usepackage{amsmath}
\usepackage{amsthm}
\usepackage{amssymb}
\usepackage{enumitem}
\usepackage{ytableau}

\usepackage{hyperref}
\usepackage{color}
\usepackage{graphicx}
\begin{document}

\newtheorem{theorem}{Theorem}[section]
\newtheorem{lemma}[theorem]{Lemma}
\newtheorem{proposition}[theorem]{Proposition}
\newtheorem{cor}[theorem]{Corollary}
\newtheorem{defn}[theorem]{Definition}
\newtheorem*{remark}{Remark}
\newtheorem{conj}[theorem]{Conjecture}

\numberwithin{equation}{section}

\newcommand{\Z}{{\mathbb Z}} 
\newcommand{\Q}{{\mathbb Q}}
\newcommand{\R}{{\mathbb R}}
\newcommand{\C}{{\mathbb C}}
\newcommand{\N}{{\mathbb N}}
\newcommand{\FF}{{\mathbb F}}
\newcommand{\fq}{\mathbb{F}_q}
\newcommand{\rmk}[1]{\footnote{{\bf Comment:} #1}}

\renewcommand{\mod}{\;\operatorname{mod}}
\newcommand{\ord}{\operatorname{ord}}
\newcommand{\TT}{\mathbb{T}}
\renewcommand{\i}{{\mathrm{i}}}
\renewcommand{\d}{{\mathrm{d}}}
\renewcommand{\^}{\widehat}
\newcommand{\HH}{\mathbb H}
\newcommand{\Vol}{\operatorname{vol}}
\newcommand{\area}{\operatorname{area}}
\newcommand{\tr}{\operatorname{tr}}
\newcommand{\norm}{\mathcal N} 
\newcommand{\intinf}{\int_{-\infty}^\infty}
\newcommand{\ave}[1]{\left\langle#1\right\rangle} 
\newcommand{\Var}{\operatorname{Var}}
\newcommand{\Prob}{\operatorname{Prob}}
\newcommand{\sym}{\operatorname{Sym}}
\newcommand{\disc}{\operatorname{disc}}
\newcommand{\CA}{{\mathcal C}_A}
\newcommand{\cond}{\operatorname{cond}} 
\newcommand{\lcm}{\operatorname{lcm}}
\newcommand{\Kl}{\operatorname{Kl}} 
\newcommand{\leg}[2]{\left( \frac{#1}{#2} \right)}  
\newcommand{\Li}{\operatorname{Li}}

\newcommand{\sumstar}{\sideset \and^{*} \to \sum}

\newcommand{\LL}{\mathcal L} 
\newcommand{\sumf}{\sum^\flat}
\newcommand{\Hgev}{\mathcal H_{2g+2,q}}
\newcommand{\USp}{\operatorname{USp}}
\newcommand{\conv}{*}
\newcommand{\dist} {\operatorname{dist}}
\newcommand{\CF}{c_0} 
\newcommand{\kerp}{\mathcal K}

\newcommand{\Cov}{\operatorname{cov}}

\newcommand{\ES}{\mathcal S} 
\newcommand{\EN}{\mathcal N} 
\newcommand{\EM}{\mathcal M} 
\newcommand{\Sc}{\operatorname{Sc}} 
\newcommand{\Ht}{\operatorname{Ht}}

\newcommand{\E}{\operatorname{E}} 
\newcommand{\sign}{\operatorname{sign}} 

\newcommand{\divid}{d} 
\newcommand{\inv}{\theta}

\newcommand{\vleq}{\rotatebox[origin=c]{90}{$\leq$}}


\theoremstyle{plain}

\newcommand\MOD{\textrm{ (mod }}
\newcommand\spann{\mathrm{span}}
\newcommand\primeE{\mathop{\vcenter{\hbox{\relsize{+2}$\mathbf{E}$}}}}
\newcommand\Tr{\mathrm{Tr}}
\newcommand\Ss{\mathcal{S}}
\newcommand\Dd{\mathcal{D}}
\newcommand\Uu{\mathcal{U}}
\newcommand\supp{\mathrm{supp}\;}
\newcommand\sgn{\mathrm{sgn}}
\newcommand\bb{\mathbb}
\newcommand\pvint{-\!\!\!\!\!\!\!\int_{-\infty}^\infty}
\def\res{\mathop{\mathrm{Res}}}
\newcommand\primesum{\sideset{}{'}\sum}
\newcommand\primeprod{\sideset{}{'}\prod}

\newtheorem{ex}[theorem]{Example}

\newcommand{\Poly}{{\operatorname{Poly}}}
\renewcommand{\mod}{\;\operatorname{mod}}
\renewcommand{\i}{{\mathrm{i}}}
\renewcommand{\d}{{\mathrm{d}}}
\renewcommand{\^}{\widehat}
\newcommand{\Sgn}{{\operatorname{Sgn}}}

\newcommand{\diag}{\operatorname{diag}}

\title
{Sums of  divisor functions in  $\fq[t]$ and matrix integrals }

\author{J.P. Keating, B.. Rodgers, E. Roditty-Gershon and Z. Rudnick}

\address{School of Mathematics, University of Bristol, Bristol BS8 1TW, UK}
\email{j.p.keating@bristol.ac.uk}

\address{Institut f\"ur Mathematik, Universit\"at Z\"urich
Winterthurerstr. 190, CH-8057 Z\"urich, Switzerland}
\email{brad.rodgers@math.uzh.ch}

\address{School of Mathematics, University of Bristol, Bristol BS8 1TW, UK}
\email{rodittye@post.tau.ac.il}

\address{Raymond and Beverly Sackler School of Mathematical Sciences,
Tel Aviv University, Tel Aviv 69978, Israel}
\email{rudnick@post.tau.ac.il}
\date{\today}

\thanks{ JPK gratefully acknowledges support under EPSRC Programme Grant EP/K034383/1
LMF: $L$-Functions and Modular Forms, a grant from Leverhulme Trust, a Royal Society Wolfson Merit Award, a Royal Society Leverhulme Senior Research Fellowship, and by the Air Force Office of Scientific Research, Air Force Material Command, USAF, under grant number FA8655-10-1-3088.  ZR is similarly grateful for support from the
European Research Council under the European Union's Seventh
Framework Programme (FP7/2007-2013) / ERC grant agreement
n$^{\text{o}}$ 320755, and from the Israel Science Foundation
(grant No. 925/14). }

\begin{abstract}
We study the mean square of sums of the $k$th divisor function $d_k(n)$ over short intervals and arithmetic progressions for the rational function
field over a finite field of $q$ elements.  In the limit as $q\rightarrow\infty$ we establish a relationship with a matrix integral over the unitary group.  Evaluating this integral enables us to
compute the mean square of the sums of $d_k(n)$ in terms of a lattice point count.  This lattice point count can in turn be calculated in terms of a certain piecewise polynomial function, which we analyse.  Our results suggest general conjectures for the corresponding classical problems over the integers, which agree with the few  cases where the answer is known.
\end{abstract}

\maketitle


 \section{Introduction}

The goal of this paper is to study the mean square of sums of
divisor functions over short intervals, for the rational function
field over a finite field, and to use the results obtained to gain
insight into the corresponding classical problem over the integers.

\subsection{Classical theory}
The   $k$-th divisor function $\divid_k(n)$  gives the number of ways of writing
a (positive) integer as a product of $k$ positive integers:
\begin{equation}\label{divdef}
\divid_k(n):=\#\{(a_1,\dots ,a_k):n=a_1\cdot \ldots \cdot a_k, \quad
a_1,\dots ,a_k\geq 1\} \;,
\end{equation}
the classical divisor function being $\divid(n)=\divid_2(n)$.

Dirichlet's divisor problem addresses the size of the remainder term
$\Delta_2(x)$ in partial sums of the divisor function:
\begin{equation}\label{def of Delta2}
\Delta_2(x):=\sum_{n\leq x} \divid_2(n)-x\Big(\log
x+(2\gamma-1)\Big)
\end{equation}
 where $\gamma$ is the Euler-Mascheroni constant. For the higher divisor
 functions one defines a remainder term $\Delta_k(x)$ similarly as the
difference between the partial sums $\sum_{n\leq x} \divid_k(n)$ and
a smooth term $xP_{k-1}(\log x)$ where $P_{k-1}(u)$ is a certain
polynomial of degree $k-1$; see, for example, \cite{Titch} Chapter XII.

The mean square of $\Delta_2(x)$ was computed by Cr\'amer
\cite{Cramer} for $k=2$, and by Tong \cite{Tong} for $k\geq 3$ 
(assuming the Riemann Hypothesis (RH) if $k\geq 4$), to be 
\begin{equation}\label{eq:Tong}
\frac 1X\int_X^{2X} \Delta_k(x)^2 dx \sim c_k X^{1-\frac 1k},
\end{equation}
for a certain constant $c_k$. Heath-Brown \cite{HB1992} showed that
$\Delta_k(x)/x^{\frac 12-\frac 1{2k}}$ has a limiting value
distribution (for $k\geq 4$ one needs to assume RH); it is
non-Gaussian.

\subsection{The divisor function in short intervals}


 Let
\begin{equation}\label{Def of Delta(x:H)}
\Delta_k(x;H) = \Delta_k(x+H)-\Delta_k(x)
\end{equation}
be the remainder term for sums of $d_k$ over short intervals
$[x,x+H]$. Our main concern is to understand its mean square.


For relatively long intervals, 
Lester \cite{Lester} proves an asymptotic (assuming RH for $k>3$) 
similar to the result  \eqref{eq:Tong}:
\begin{equation}\label{ms for large H}
\frac 1X\int_X^{2X} \Big(\Delta_k(x,H) \Big)^2 dx \sim 2c_k
X^{1-\frac 1k}, \quad X^{1-\frac 1k+o(1)}< H<  X^{1-o(1)}
\end{equation}

The interesting range for us is that of shorter intervals:
$H<X^{1-\frac 1k}$. For $k=2$, Jutila \cite{Jutila1984}, Coppola and
Salerno \cite{CS}, and Ivi\'c \cite{Ivic Bordeaux 2009,  Ivic
RramanujanJ} show that, for $X^\epsilon<H< X^{1/2-\epsilon}$,  the
mean square of $\Delta_2(x,H)$ is asymptotically equal to
\begin{equation}\label{Ivic}
\frac 1X\int_X^{2X} \Big(\Delta_2(x,H)\Big)^2 dx \sim HF_3(\log
\frac{X^{1/2}}{H})
\end{equation}
for a certain cubic polynomial $F_3$. In that regime, Lester and
Yesha \cite{LY} showed that $\Delta_2(x,H)$, normalized to have unit
mean-square using \eqref{Ivic}, has a Gaussian value distribution,
at least for a narrow range of $H$ below $X^{1/2}$, 
the conjecture being that this should hold for
$X^\epsilon<H<X^{1/2-\epsilon}$ for any $\epsilon>0$.

For $k\geq 3$, Milinovich and Turnage-Butterbaugh \cite[p. 182]{MTB}
give an upper bound,  assuming  RH, of
\begin{equation}\label{eqMTB}
\frac 1X\int_X^{2X} \Big(\Delta_k(x,H) \Big)^2 dx \ll H(\log
X)^{k^2+o(1)}\;, \quad X^\epsilon<H<X^{1-\epsilon}
\end{equation}

In concurrent work, Lester \cite{Lester}  shows that for $k\geq 3$,
assuming the Lindel\"of Hypothesis, if $h(x)=(\frac{x}{X})^{1-\frac 1k} X^\delta$, 
\begin{equation}\label{Lester}
\frac 1X\int_X^{2X} \Big(\Delta_k(x,h(x)) \Big)^2
dx \sim a_k \cdot   \frac{k^{k^2-1}}{\Gamma(k^2)}(1-\frac
1k-\delta)^{k^2-1}  \cdot  \frac{2^{2-\frac 1k}-1}{2-\frac
1k} X^\delta   \cdot  (\log  X)^{k^2-1},
\end{equation}
provided  $1-\frac 1{k-1}<\delta<1-\frac 1k$, 
where
\begin{equation}\label{Lester const}
 a_k = \prod_p \Big\{ (1- \frac 1p)^{k^2}\sum_{j=0}^\infty
\left(\frac{\Gamma(k+j)}{\Gamma(k)j!}\right)^2 \frac 1{p^j} \Big\} =
\prod_p \Big\{(1-\frac 1p)^{(k-1)^2}\sum_{j=0}^{k-1}
\binom{k-1}{j}^2p^{-j} \Big\}.
\end{equation}
For $k=3$ and $\frac 7{12} <\delta<\frac 23$, 
the result is unconditional.


\subsection{A Conjecture}\label{Conj-intro}
 We did not find any conjecture in the literature
 for the order of growth of the mean-square of $\Delta_k(x;H)$ for small $H$.
Based on Theorems~\ref{lattice_count} and \ref{thm:Asymp of I}, we believe
the following:
\begin{conj}\label{conj:mean square SI}
If $0<\delta <1-\frac 1{k}$ is fixed, then for $H=X^\delta$,
\begin{equation}
\frac 1X\int_X^{2X} \Big(\Delta_k(x,H) \Big)^2 dx \sim  a_k \mathcal
P_k(\delta) H(\log X)^{k^2-1}\;,\quad X\to \infty
\end{equation}
where $a_k$ is given by \eqref{Lester const}, and $\mathcal
P_k(\delta)$ is a piecewise polynomial function of $\delta$, of
degree $k^2-1$, given by
\begin{equation}\label{def of P_k}
\mathcal P_k(\delta) =(1-\delta)^{k^2-1} \gamma_k(\frac 1{1-\delta})
\;.
\end{equation}
Here
\begin{equation} \label{def of gamma2}
\gamma_k(c) = \frac{1}{k!\, G(1+k)^2} \int_{[0,1]^k} \delta_c(w_1 +
\ldots + w_k) \prod_{i< j}(w_i-w_j)^2\, d^k w,
\end{equation}
where $\delta_c(x) =\delta(x-c)$ is the delta distribution
translated by $c$, and $G$ is the Barnes $G$-function, so that for
positive integers $k$, $G(1+k) = 1!\cdot 2! \cdot 3! \cdots (k-1)!$.
\end{conj}

For $1-\frac 1{k-1}<\delta<1-\frac 1k$,  \eqref{def of P_k}  reduces
to the simpler form
\begin{equation}\label{Les}
\mathcal P_k(\delta) = \frac{k^{k^2-1}}{(k^2-1)!}(1-\frac
1k-\delta)^{k^2-1}\;,
\end{equation}
rendering visible the compatibility of Conjecture~\ref{conj:mean
square SI} with Lester's result \eqref{Lester}, which  corresponds
to taking $H=X^\delta$, with $\delta$ in this range.
Note that in \eqref{Lester}, the length of the interval
$h(x)=(\frac{x}{X})^{1-\frac 1k}X^\delta$ varies with $x$, and this slight difference
in conventions is responsible for the factor of $\frac{2^{2-\frac
1k}-1}{2-\frac 1k}$ in \eqref{Lester}, since the mean value of
$h(x)$ over $[X,2X]$ is
\begin{equation}
\frac 1X\int_X^{2X} h(x)dx = \frac{2^{2-\frac 1k}-1}{2-\frac
1k} X^\delta  \;.
\end{equation}

As will be explained later, $\gamma_k(c)$ is a piecewise polynomial function of $c$, satisfying $\gamma_k(c)=\gamma_k(k-c)$, that relates to the asymptotics of a lattice counting problem (Theorem~\ref{lattice_count}).  This lattice counting problem itself emerges from the evaluation of a matrix integral over the unitary group.  We also note that it is possible to write down conjectures for the lower order terms in the asymptotic expansion (\ref{conj:mean square SI}): the right-hand side is, up to terms that are $o(1)$, a polynomial in $\log X$ whose coefficients can be computed.  This is explained in Section~\ref{Conj}.

\subsection{Divisor functions in  $\fq[x]$}
We study the  problem of the sum of   divisor functions
$\divid_k(f)$ over short intervals  for $\fq[x]$.   The   divisor
functions $\divid_k(f)$ for a monic polynomial $f$ are defined in analogy to \eqref{divdef} and give the number
of decompositions $f=f_1 f_2\dots f_k$ with $f_i $ monic. In
particular $\divid_2=\divid$ is the classical divisor function.  

We denote by $\EM_n$ the set of monic polynomials of degree $n$.  A ``short interval" in $\fq[x]$ is a set of the form
\begin{equation}
I(A;h) = \{f:||f-A||\leq q^h\}
\end{equation}
where  $A\in \EM_n$ has degree $n$, $0\leq h\leq n-2$ and
the norm is
\begin{equation}
||f||:=\#\fq[t]/(f) = q^{\deg f}\;.
\end{equation}
The cardinality of such a short interval is
\begin{equation}
\#I(A;h)=q^{h+1}=:H \;.
\end{equation}

Set
\begin{equation}
\EN_{\divid_k}(A;h):=\sum_{f\in I(A;h)} \divid_k(f) \;.
\end{equation}
The mean value is (c.f.~\cite{Andrade}) 
\begin{equation}
\frac 1{q^n}\sum_{A\in \EM_n} \EN_{\divid_k} (A;h) = \frac
{q^{h+1}}{q^n}\sum_{f\in \EM_n} \divid_k(f)
=q^{h+1}\binom{n+k-1}{k-1} \;.
\end{equation}

In analogy with \eqref{def of Delta2} and \eqref{Def of Delta(x:H)} we
set
\begin{equation}\label{Def of Delta_k}
\Delta_k(A;h):=\EN_{\divid_k}(A;h)-q^{h+1}\binom{n+k-1}{k-1} \;.
\end{equation}
We will show below   (Theorem~\ref{thm:sec SI})  that
\begin{equation}
\EN_{\divid_k}(A;h)= q^{h+1}\binom{n+k-1}{k-1}, \quad
h>(1-\frac 1k)n-1
\end{equation}
so that $\Delta_k(A;h)= 0$ vanishes identically for $h>(1-\frac
1k)n-1$. The corresponding range over the integers is $X^{1-\frac
1k} <H<X$, where we have a bound of $O(X^{1-\frac 1k} )$ for the
mean square, see \eqref{ms for large H}.


Our principal result gives the mean square of $\Delta_k(A;h)$
  (which is the variance  of $\EN_{\divid_k}(A;h)$), in the limit $q\to \infty$, in terms of a matrix integral.  Let $U$ be an $N\times N$ matrix. The {\em secular coefficients}
$\Sc_j(U)$ are the coefficients  of the characteristic polynomial of
$U$:
\begin{equation}
\det(I+xU) =\sum_{j=0}^N \Sc_j(U) x^j
\end{equation}
Thus $\Sc_0(U)=1$, $\Sc_1(U) = \tr U$, $\Sc_N(U) = \det U$.  The secular coefficients are the elementary symmetric functions in
the eigenvalues $\lambda_1,\dots,\lambda_N$ of $U$:
\begin{equation}
\Sc_r(U) = \sum_{1\leq i_1<\dots<i_r \leq N} \lambda_{i_1}\cdot
\dots \cdot \lambda_{i_r}
\end{equation}
and give the character of the exterior power representation on
$\wedge^j \C^N$:
\begin{equation}
\Sc_j(U) = \tr \wedge^j(U)
\end{equation}

It is well known that $\wedge^j $ are distinct irreducible
representations of the unitary group $U(N)$, and hence one gets the
mean values
\begin{equation}
\int_{U(N)} \Sc_j(U)dU=0, \quad j=1,\dots, N
\end{equation}
and
\begin{equation}
\int_{U(N)} \Sc_j(U) \overline{\Sc_k(U)} dU = \delta_{j,k},
\end{equation}
where the integrals are with respect to the Haar probability measure

Define the matrix integrals over the group $U(N)$ of $N\times N$
unitary matrices
\begin{equation}\label{def of I}
I_k(m;N):=\int_{U(N)} \Big| \sum_{\substack{j_1+\dots + j_k=m\\
0\leq j_1,\dots, j_k \leq N}}\Sc_{j_1}(U)  \dots
\Sc_{j_k}(U)\Big|^2 dU\;.
\end{equation}
Then the variance
\begin{equation}
\Var(\EN_{\divid_k}):=\frac 1{q^n}\sum_{A\in \EM_n} |\Delta_k(A;h)|^2
\end{equation}
satisfies
\begin{theorem}\label{thm:sec SI}
If  $0\leq h\leq \min(n-5, (1-\frac 1k)n-2)$, then as $q \to \infty$
\begin{equation}
\Var(\EN_{\divid_k}) =H \cdot I_k(n;n-h-2) +O\Big(\frac{H}{\sqrt{q}}
\Big)\;.
\end{equation}
In the remaining cases,
\begin{equation}
\Var(\EN_{\divid_k}) =O\Big(\frac{H}{\sqrt{q}} \Big),\quad h=\lfloor
(1-\frac 1k)n\rfloor-1
\end{equation}
and
\begin{equation}
\Var(\EN_{\divid_k}) =0,\quad    \lfloor (1-\frac 1k)n\rfloor \leq
h\leq n \;.
\end{equation}
\end{theorem}
\noindent In the case $(1-\frac 1{k-1})n<h+2\leq (1-\frac 1k)n$, the
matrix integral takes a simple form, c.f.~Theorem~\ref{Thm: Ik(m,N) for large m} below.




\subsection{Matrix integrals}
For
$(k-1)N<m<kN$, we obtain a simple formula for the matrix integral:
\begin{theorem}\label{Thm: Ik(m,N) for large m}
For $(k-1)N<m<kN$,
\begin{equation}\label{Ik(m,N) for large m}
I_k(m;N) = \binom{kN-m+k^2-1}{k^2-1}\;.
\end{equation}
\end{theorem}

We are also able to give a closed form, albeit more complicated,
formula for the matrix integral for any range of the parameters, in
terms of a lattice point count:

\begin{theorem}
\label{lattice_count}
$I_k(m;N)$ is equal to the count of lattice points $x= (x_i^{(j)})\in\mathbb{Z}^{k^2}$ satisfying each of the relations
\begin{enumerate}
\item $0 \leq x_i^{(j)} \leq N$ for all $1\leq i,j \leq k$
\item $x_1^{(k)} + x_2^{(k-1)} + \cdots + x_k^{(1)} = kN - m$, and
\item $x\in A_k$, 
\end{enumerate} 
where $A_k$ is the collection of $k \times k$ matrices whose entries satisfy the following system of inequalities,
$$
\begin{matrix}
x_1^{(1)}& \leq & x_1^{(2)}& \leq & \cdots & \leq & x_1^{(k)}\\
\vleq & & \vleq & & & &\vleq \\
x_2^{(1)}& \leq &  x_2^{(2)}& \leq &\cdots & \leq & x_2^{(k)} \\
\vleq & & \vleq & & & &\vleq \\
\vdots& & \vdots & & \ddots & & \vdots \\
\vleq & & \vleq & & & &\vleq \\
x_k^{(1)}& \leq & x_k^{(2)}& \leq & \cdots & \leq & x_k^{(k)}\\
\end{matrix}
$$
\end{theorem}
We note in passing that the above count of lattice points also may be interpreted as a count of plane partitions (see \cite{St}, Section 7.20 for an introduction to the latter).

For the standard divisor function ($k=2$), 
if $h\leq n/2-2$ and $n\geq 5$   we thus find that as $q\to \infty$,
\begin{equation}
\frac 1{q^n}\sum_{A\in \EM_n} |\Delta_2(A;h)|^2 \sim  H \frac{
(n-2h+5) (n-2h+6) ( n-2h+7)}6 \;.
\end{equation}
 This is consistent with  \eqref{Ivic}, which  leads us to expect
a cubic polynomial in $\frac n2-h$.

For the range $(1-\frac 1{k-1})n<h+2<(1-\frac 1k)n$, \eqref{Ik(m,N)
for large m} gives
\begin{equation}\label{intermediate regime}
\begin{split}
\frac 1{q^n}\sum_{A\in \EM_n} |\Delta_k(A;h)|^2 &\sim  H  \binom{
k(n-h-2)-n+k^2-1}{k^2-1} \\
&= HQ_{k^2-1}\Big((1-\frac 1k)n-(h+1)\Big)\;,
\end{split}
\end{equation}
where $Q_{k^2-1}(u)$ is a polynomial of degree $k^2-1$, given by
\begin{equation}\label{int reg asym}
Q_{k^2-1}(u):= \frac{\prod_{j=1}^{k^2-1}( k(u-1)+j) }{\Gamma(k^2)} =
\frac{k^{k^2-1}}{\Gamma(k^2)} u^{k^2-1} + \ldots
\end{equation}
As this range corresponds to $X^{1-\frac{1}{k-1}}<H<X^{1-\frac 1k}$
over the integers, the result \eqref{intermediate regime},
\eqref{int reg asym} is comparable with Lester's result \eqref{Lester} (c.f.~the remark after \eqref{Les}).

We use these results to model the situation over the integers
for the range $H<X^{1-\frac 1{k-1}}$, leading to
Conjecture~\ref{conj:mean square SI}. To do so we derive asymptotics
of $I_k(m;N)$ for $m\approx N$: 
\begin{theorem}\label{thm:Asymp of I}
\label{integral} Let $c := m/N$. Then for $c \in [0,k]$,
\begin{equation}
\label{thm:poly_approx} I_k(m;N) = \gamma_k(c) N^{k^2-1} +
O_k(N^{k^2-2}),
\end{equation}
with $\gamma_k(c) $ defined by \eqref{def of gamma2}.
%
\end{theorem}

The matrix integral satisfies a functional equation $I_k(m;N) = I_k(kN-m;N)$ (see Lemma~\ref{functional eq}), from which it follows that 
\begin{equation}
\gamma_k(c) = \gamma_k(k-c)\;.
\end{equation}
It follows from an alternative analysis of $I_k(m;N)$ that we also have
\begin{theorem}\label{thm:PP} 
\label{integral analytic}
\begin{equation}
\label{gamma analytic}
\gamma_k(c) =\sum_{0\leq \ell <c}\binom{k}{\ell}^2 (c-\ell)^{(k-\ell)^2+\ell^2-1} g_{k,\ell} (c-\ell)
\end{equation}
where $g_{k,\ell}(c-\ell)$ are (complicated) polynomials in $c-\ell.$ 
\end{theorem} 
\noindent and from this that

\begin{cor}
For a fixed $k$, $\gamma_k(c)$ is a piecewise polynomial function of $c$.  Specifically,  it is a fixed polynomial for $r\leq c<r+1$ ($r$ integer), and each time the value of $c$ passes through an integer it
becomes a different polynomial.
\end{cor}

For example,
$$
\gamma_2(c) = \frac {1}{2!} \int_{\substack{0\leq w_1\leq
1\\0\leq c-w_1\leq 1}} (w_1-(c-w_1))^2\,dw_1 = \begin{cases}
\frac{c^3}{3!},& 0\leq c\leq 1\\  & \\ \frac {(2-c)^3}{3!},&1\leq
c\leq 2\end{cases}
$$
and similarly 
\begin{equation*}\label{eq:rod3easy}
\gamma_3(c) =
\begin{cases} \frac 1{8!}c^8,&0<c<1\\
\frac 1{8!}(3-c)^8,&2<c<3
\end{cases}
\end{equation*}
while for $1<c<2$ we get
\begin{equation*}\label{eq:rod3mid}
\gamma_3(c)=\frac 1{8!}\Big(-2 c^8+24 c^7-252 c^6+1512 c^5-4830
c^4+8568 c^3-8484 c^2+4392 c-927\Big)
\end{equation*}


\subsection{Arithmetic progressions}
A similar theory can be developed for sums of divisor functions
along arithmetic progressions,  see \S~\ref{secAP}.

\section{The divisor functions in short intervals}\label{secSI}

Our first goal is to provide proofs for Theorem~\ref{thm:sec SI} and the other
results on sums of $\divid_k$ in short intervals.

\subsection{An expression for the variance}
To begin the proof of Theorem~\ref{thm:sec SI}, we express the
variance of the short interval sums $\EN_{\divid_k}$ in terms of
sums of divisor functions, twisted by primitive even Dirichlet
characters. Recall that a Dirichlet character is {\em even} if
$\chi(cf) = \chi(f)$ for all $c\in \fq^\times$, and is {\em odd}
otherwise. The number of even characters modulo $T^{n-h}$ is
$\Phi(T^{n-h})/(q-1) =q^{n-h-1}$ (see e.g. \cite[\S3.3]{KR}). We
denote by $\Phi_{ev}^*(T^{n-h})=q^{n-h-2}(q-1)$ the number of
primitive even characters modulo $T^{n-h}$.

For a Dirichlet character $\chi$ modulo $T^{n-h}$, set
\begin{equation}\label{def of M(n)}
\EM(n;\divid_k\chi) := \sum_{f\in \mathcal M_n} \divid_k(f)\chi(f)
\;.
\end{equation}

\begin{lemma} 
As $q \to \infty$
\begin{equation}\label{reduce var to M}
\Var(\EN_{\divid_k})    = \frac{H}{q^n}
\frac{1}{\Phi_{ev}^*(T^{n-h})} \sum_{\substack{\chi \bmod
T^{n-h}\\\chi \;{\rm even \; primitive}}} |\EM(n;\divid_k\chi)|^2  +
O\Big(\frac{H}{\sqrt{q}} \Big)
\end{equation}
\end{lemma}
\begin{proof}
 To compute the variance, we use
\cite[Lemma 5.4]{KRsf}  which gives an expression for the variance
of sums over short intervals of certain arithmetic functions
$\alpha$ which are ``even" ($\alpha(cf)=\alpha(f)$ for $c\in
\fq^\times$), multiplicative, and symmetric under the  map
$f^*(t):=t^{\deg f}f(\frac 1t)$, in the sense that
\begin{equation}\label{inv under*}
\alpha(f^*) = \alpha(f), \quad {\rm if}\;  f(0)\neq 0 \;.
\end{equation}
 Since the divisor functions $\divid_k$ clearly satisfy all
these conditions, we may use  \cite[Lemma 5.3]{KRsf} (compare
\cite[\S4.5]{KR}) to obtain
\begin{multline}\label{var short int d}
\Var(\EN_{\divid_k})    = \frac{H}{q^n} \sum_{m_1,m_2=0}^n
\divid_k(T^{n-m_1}) \overline{\divid_k(T^{n-m_2})} \\
\times \frac{1}{\Phi_{ev}(T^{n-h})} \sum_{\substack{\chi \bmod
T^{n-h}\\\chi\neq \chi_0\, {\rm even}}} \EM(m_1;\divid_k\chi)
 \overline{\EM(m_2;\divid_k\chi)}
\end{multline}

To compute $\EM(n;\divid_k\chi)$, we introduce the generating function
\begin{equation}
\sum_{m=0}^\infty \EM(m;\divid_k\chi)u^m = \sum_{f\; \rm monic}
\chi(f)\divid_k(f)u^{\deg f} = L(u,\chi)^k
\end{equation}
Hence $\EM(n;\divid_k\chi)$ is the coefficient of $u^n$ in
$L(u,\chi)^k$. Now for even $\chi\neq \chi_0$, we write $L(u,\chi) =
(1-u)P(u,\chi)$ and by the Riemann Hypothesis for curves, $P(u,\chi)
= \prod_{j=1}^{n-h-2}(1-\alpha_ju)$ with the inverse zeros
satisfying $|\alpha_j|\leq \sqrt{q}$. Hence we have an a-priori
bound
\begin{equation}\label{apriori bd}
|\EM(m;\divid_k\chi)|\ll_{n,k} q^{m/2}\;.
\end{equation}

Therefore in the sum \eqref{var short int d}, the terms with
$m_1+m_2<2n$ (i.e. $(m_1,m_2)\neq (n,n)$) will contribute
$O(\frac{H}{q^n}q^{n-\frac 12})=O(H/\sqrt{q})$  (the coefficients
$\divid_k(T^{n-m}) = \binom{n-m+k-1}{k-1}$ do not depend on $q$).
Thus
\begin{equation}
\Var(\EN_{\divid_k})    = \frac{H}{q^n}\frac{1}{\Phi_{ev}(T^{n-h})}
\sum_{\substack{\chi \bmod T^{n-h}\\\chi\neq \chi_0\, {\rm even}}}
|\EM(n;\divid_k\chi)|^2  + O\Big(\frac{H}{\sqrt{q}} \Big) \;.
\end{equation}

For the same reason, the non-primitive even characters, whose number
is $\ll \Phi_{ev}(T^{n-h})/q$ (see \cite[\S3.3]{KR}), contribute
$O(H/q)$ to the variance. Thus we are left with
\begin{equation}
\Var(\EN_{\divid_k})    = \frac{H}{q^n}
\frac{1}{\Phi_{ev}^*(T^{n-h})} \sum_{\substack{\chi \bmod
T^{n-h}\\\chi \;{\rm even \; primitive}}} |\EM(n;\divid_k\chi)|^2  +
O\Big(\frac{H}{\sqrt{q}} \Big)
\end{equation}
\end{proof}

\subsection{The sums $\EM(n;\divid_k\chi)$}
We need some information on $\EM(n;\divid_k\chi)$ for $\chi$ even
and primitive. By the Riemann Hypothesis (Weil's theorem), for
$\chi$ even and primitive modulo $T^{n-h}$, we write
\begin{equation}
L(u,\chi) = (1-u)\det(I-uq^{1/2}\Theta_\chi)
\end{equation}
with $\Theta_\chi\in U(n-h-2)$ a unitary matrix of size $n-h-2$.

\begin{lemma}\label{lem:compute M(n)}
For $\chi$ even, primitive modulo $T^{n-h}$:
\begin{itemize}
\item If $n\leq k(n-h-2)$, that is $h\leq (1-\frac 1k)n-2$, then
\begin{equation}
 \EM(n,\chi \divid_k) = (-1)^nq^{n/2}\sum_{\substack{j_1+\ldots+j_k=n\\
j_1,\dots ,j_k\leq n-h-2}} \prod \Sc_{j_i}(\Theta_\chi) +
O(q^{n-\frac 12})
\end{equation}
\item For $k(n-h-2)<n\leq k(n-h-1)$, 
i.e. $h=\lfloor (1-\frac 1k)n\rfloor-1$ we get
\begin{equation}
\EM(n,\divid_k\chi) = O(q^{\frac{n-1}2})
\end{equation}

\item For $n>k(n-h-1)$, that is $h>(1-\frac 1k)n-1$, we get $\EM(n,\divid_k\chi) = 0$.
\end{itemize}
\end{lemma}
\begin{proof}
For a primitive even character, the L-function is
\begin{equation}
\begin{split}
L(u,\chi) &= (1-u) \det (I-u\sqrt{q}\Theta_\chi)  \\
&= (1-u)\sum_{j=0}^{n-h-2} (-1)^jq^{j/2} \Sc_{j}(\Theta_\chi)u^j
\end{split}
\end{equation}
To simplify notation in the calculations below, we write
\begin{equation}
N=n-h-2 \;,
\end{equation}
\begin{equation}
a_j = (-1)^jq^{j/2}\Sc_{j}(\Theta_\chi), \quad 0\leq j\leq N, \quad
a_{-1}=0=a_{N+1}\;,
\end{equation}
 and set
\begin{equation}
b_j:=a_j-a_{j-1}, \quad j=0,\dots, N+1
\end{equation}
so that for $\chi$ even, primitive
\begin{equation}
L(u,\chi) = \sum_{j=0}^{N+1} b_j u^j
\end{equation}
and
\begin{equation}
L(u,\chi)^k = \sum_{j_1,\dots, j_k=0}^{N+1} b_{j_1}\cdot  \ldots
\cdot b_{j_k} u^{j_1+\dots +j_k}\;.
\end{equation}
Therefore the coefficient of $u^n$ in the expansion of $L(u,\chi)^k$
for $\chi$ even and primitive is
\begin{equation}
\EM(n,\divid_k\chi) = \sum_{\substack{j_1+\ldots + j_k = n\\0\leq
j_1,\ldots, j_k\leq N+1}}b_{j_1}\cdot  \ldots \cdot b_{j_k}
\end{equation}

Note that
\begin{equation} b_j = a_j + O(q^{\frac{j-1}2}), \quad 0 \leq
j \leq N
\end{equation}
while
\begin{equation}
|b_{N+1}|\ll q^{N/2}\;.
\end{equation}
Hence for an $k$-tuple $(j_1,\ldots,j_k)$ where one of the $j_i=N+1$
we have an upper  bound
\begin{equation}
| b_{j_1} \cdot \ldots \cdot b_{j_k} | \ll q^{\frac{n-1}2}\;.
\end{equation}

Thus if $n>kN$, and $j_1+\dots+j_k=n$, there is at least one index
$i$ so that $j_i=N+1$ and in that case
\begin{equation}
|\EM(n,\divid_k\chi)|\ll q^{\frac{n-1}2}, \quad kN<n\leq k(N+1) \;.
\end{equation}

For $n\leq kN$, there will always be an $k$-tuple of $0\leq
j_1,\ldots,j_k\leq N$ with $j_1+\dots+j_k=n$, and so for $n\leq
kN$
\begin{equation}
\begin{split}
\EM(n,\divid_k\chi)&= \sum_{\substack{j_1+\ldots+j_k=n\\0\leq
j_1,\dots ,j_k\leq N}} b_{j_1}\cdot \ldots \cdot b_{j_k}  +
O\Big(q^{\frac{n-1}2} \Big) \\
&=(-1)^nq^{n/2} \sum_{\substack{j_1+\ldots+j_k=n\\0\leq j_1,\dots
,j_k\leq N}} \Sc_{j_1}(\Theta_\chi)\cdot \ldots \cdot
\Sc_{j_k}(\Theta_\chi) + O\Big(q^{\frac{n-1}2} \Big) \;.
\end{split}
\end{equation}
This concludes the proof.
\end{proof}

\subsection{Proof of Theorem~\ref{thm:sec SI}}
Inserting Lemma~\ref{lem:compute M(n)} into \eqref{reduce var to M}
we find that for $h\leq (1-\frac 1k)n-2$,
\begin{multline}
\Var(\EN_{\divid_k}) =\frac{ H }{\Phi_{ev}^*(T^{n-h})} \sum_{\substack{\chi
\bmod T^{n-h}\\  {\rm even \; primitive}}}
\Big| \sum_{\substack{j_1+\ldots+j_k=n\\
0\leq j_1,\dots ,j_k\leq n-h-2}}   \Sc_{j_1}(\Theta_\chi)\cdot \ldots \cdot  \Sc_{j_k}(\Theta_\chi) \Big|^2 \\
+ O\Big(\frac{H}{\sqrt{q}} \Big)
\end{multline}

We now apply Katz's equidistribution theorem for primitive even
characters modulo $T^N$ \cite{KatzKR2}, which says that the
corresponding Frobenii $\Theta_\chi$ are equidistributed in the
projective unitary group $PU(N-2)$ if $N \geq 5$, to replace the average over
primitive even characters by a matrix integral over $PU(n-h-2)$,
with an error of $O(1/\sqrt{q})$. This gives
\begin{equation}
\Var(\EN_{\divid_k}) =H \cdot I_k(n;n-h-2) +O\Big(\frac{H}{\sqrt{q}}
\Big),\quad  0\leq h\leq \min(n-5, (1-\frac 1k)n-2)
\end{equation}
which proves the main statement of our Theorem.

In the remaining cases, Lemma~\ref{lem:compute M(n)} gives
\begin{equation}
\Var(\EN_{\divid_k}) =O\Big(\frac{H}{\sqrt{q}} \Big),\quad   h=\lfloor (1-\frac
1k)n\rfloor-1
\end{equation}
and
\begin{equation}
\Var(\EN_{\divid_k}) =0,\quad    \lfloor (1-\frac 1k)n\rfloor \leq h\leq n \;.
\end{equation}
This concludes the proof of Theorem~\ref{thm:sec SI}.  \qed

\section{The divisor function in arithmetic progressions}\label{secAP}

\subsection{Arithmetic progressions}
We now turn to   sums of  divisor functions over arithmetic
progressions. Set
\begin{equation}
\ES_{\divid_k}(A)=\ES_{\divid_k;X;Q}(A) = \sum_{\substack{n\leq
X\\n=A\bmod Q}}\divid_k(n)
\end{equation}
For the standard divisor function ($k=2$), it is known that
individually, if $Q<X^{2/3-\epsilon}$ then
\begin{equation}
\ES_{\divid_2}(A) = \frac{X p_Q(\log X)}{\Phi(Q)} + O(X^{1/3+o(1)})
\end{equation}
for some linear polynomial $p_Q$. This is apparently due to Selberg
(unpublished). For recent work on asymptotics of sums of $\divid_3$
over arithmetic progressions, see \cite{FKM} and the literature
cited therein.

The variance $\Var(\ES_{\divid_2;X;Q})$  of $\ES_{\divid_2}$
has been studied by Motohashi \cite{Motohashi},
Blomer \cite{Blomer}, Lau and Zhao \cite{Lau Zhao}, the result being
\cite{Lau Zhao} (we assume $Q$ prime for simplicity):

i) If $1\leq Q<X^{1/2+\epsilon}$ then
\begin{equation}
\Var(\ES_{\divid_2;X;Q}) \ll X^{1/2} + (\frac XQ)^{2/3+\epsilon}
\end{equation}

ii) For $X^{1/2}<Q<X$,
\begin{equation}\label{eq:Lau Zhao}
\Var(\ES_{\divid_2;X;Q}) = \frac XQ p_3(\log \frac{Q^2}X) + O((\frac
XQ)^{5/6}(\log X)^3 )
\end{equation}
where $p_3$ is a polynomial of degree $3$ with positive leading
coefficient. See also the recent papers by Fouvry, Ganguli,
Kowalski, Michel \cite{FGKM} and by Lester and Yesha \cite{LY}
discussing higher moments.

For $k\geq 3$, Kowalski and Ricotta \cite{KowRic} considered smooth
analogues of the divisor sums $\ES_{\divid_k;X;Q}(A)$, and among
other things computed the variance\footnote{The statement of
\cite[Theorem A]{KowRic}, which deals with all moments, includes a
term which is not small for the second moment; however the actual
proof, see \cite[equation 9.8 and below]{KowRic} does give a good
remainder.}
 for $Q^{k-\frac 12+\epsilon}<X<Q^{k-\epsilon}$.


We turn to $\fq[x]$. For $Q\in \fq[x]$ squarefree of degree at least $2$, and
$A$ co-prime to $Q$, set
\begin{equation}
\ES_{\divid_k,n,Q}(A) := \sum_{\substack{f\in \mathcal M_n\\f=A\bmod
Q}} \divid_k(f) \;.
\end{equation}
Our main result here concerns the variance
\begin{equation}
\Var_Q(\ES_{\divid_k,n,Q}):=\frac 1{\Phi(Q)}\sum_{\substack{A\bmod Q\\
\gcd(A,Q)=1}}  \Big| \ES_{\divid_k,n,Q}(A) -\langle \ES_{\divid_k,n,Q}\rangle\Big|^2
\end{equation}
in the range $n\leq k(\deg Q-1)$.
\begin{theorem}\label{thm var div AP}
 If $Q$ is squarefree, and $n\leq k(\deg Q-1)$, then the variance is given by
\begin{equation}
\lim_{q\to \infty} \frac{\Var_Q(\ES_{\divid_k,n,Q})}{  q^n/|Q|} =
I_k(n;\deg Q-1)
\end{equation}
\end{theorem}

In particular for the classical divisor function $\divid=\divid_2$,
we get a result consistent with \eqref{eq:Lau Zhao}:
\begin{cor}\label{cor divisors AP}
If  $Q$ is squarefree, $\deg Q\geq 2$ and $n\leq 2(\deg Q-1)$, then
\begin{equation}
\lim_{q\to \infty} \frac{\Var_Q(\ES_{\divid_2,n,Q})}{  q^n/|Q|} =
\begin{cases}
Pol_3(n), & n\leq \deg Q-1 \\ Pol_3(2(\deg Q-1)-n),& \deg Q\leq
n\leq 2(\deg Q-1)
\end{cases}
\end{equation}
where $Pol_3(x) = \binom{x+3}{3} = (x+1)(x+2)(x+3)/6$.
\end{cor}

As in the short interval case, we are led to a conjecture on the
asymptotics of the variance over the integers. For simplicity, we
stick with the case that the the modulus $Q$ is prime: 
\begin{conj}\label{conj:AP}
For $Q$ prime, $Q^{1+\epsilon}<X<Q^{k-\epsilon}$, as $X\to \infty$,
$$
\Var(\ES_{\divid_k;X;Q})\sim \frac XQ a_k \gamma_k(\frac{\log
X}{\log Q})(\log Q)^{k^2-1}
$$
where $a_k$ is given by \eqref{Lester const} and $\gamma_k(c)$ is
given by \eqref{def of gamma2}.
\end{conj}

\subsection{Proof of Theorem~\ref{thm var div AP}}\label{sec:div AP}
We start with  the following expansion, using the orthogonality
relation for Dirichlet characters to pick out an arithmetic
progression \cite[\S4.1]{KRsf}:
\begin{equation}\label{expand ES}
\ES_{\divid_k,n,Q}(A) = \frac 1{\Phi(Q)} \sum_{\substack{f\in \mathcal M_n\\
 (f,Q)=1}} \divid_k(f)  + \frac{1}{\Phi(Q)} \sum_{\chi\neq
\chi_0}\overline{\chi(A)} \EM(n;\divid_k\chi)
\end{equation}
where $\EM(n;\divid_k\chi)$, given by \eqref{def of M(n)},
is the coefficient of $u^n$ in the expansion of $L(u,\chi)^k$.
Since $L(u,\chi)$ is a polynomial of degree $\leq \deg Q-1$ for
$\chi\neq \chi_0$, we see that $\ES_{\divid_k,n,Q}$ is independent of $A$ for $n>k(\deg Q-1)$:
\begin{equation} \ES_{\divid_k,n,Q}(A) =
\frac 1{\Phi(Q)} \sum_{\substack{f\in \mathcal M_n\\
 (f,Q)=1}} \divid_k(f)\sim \frac{q^n\binom{n+k-1}{k-1}}{\Phi(Q)}  \;.
 \end{equation}
Thus for any $n$, the mean value (averaging over $A$ coprime to $Q$)
is
\begin{equation}
\ave{\ES_{\divid_k,n,Q}}  \sim \frac{q^n\binom{n+k-1}{k-1}}{\Phi(Q)} \;.
 \end{equation}

The interesting range is $n\leq k(\deg Q-1)$, which we assume from
now on. To compute the variance, we use \eqref{expand ES} and the
orthogonality relations for Dirichlet characters as in \cite{KR,
KRsf} to find
\begin{equation}
\Var(\ES_{\divid_k,n,Q}) = \frac 1{\Phi(Q)^2} \sum_{\chi\neq \chi_0}
|\EM(n;\divid_k\chi)|^2 \;.
\end{equation}

We first dispose of the contribution of even characters, whose
number is $\Phi_{\rm ev}(Q) = \Phi(Q)/(q-1)$:
As in \eqref{apriori bd},  we have an a-priori bound for $\chi\neq
\chi_0$
\begin{equation}
|\EM(n,\divid_k \chi)|\ll_n q^{n/2} \;.
\end{equation}
Therefore the even characters  contribute at most
\begin{equation}
\ll_n \frac 1{\Phi(Q)^2} \Phi_{\rm ev}(Q) q^{n} \ll \frac 1q
\frac{q^n}{\Phi(Q)}\;,
\end{equation}
which is negligible relative to the main term that we find which is
of order $q^n/\Phi(Q)$. The same argument bounds the contribution of
odd non-primitive characters if $Q$ is non-prime.  Thus
\begin{equation}\label{var intermed odd}
\Var(\ES_{\divid_k,n,Q}) = \frac 1{\Phi(Q)^2} \sum_{\substack{\chi\;
\rm odd\; and \; primitive}}|\EM(n;\divid_k\chi)|^2  + O\Big(\frac 1q
\cdot \frac{q^n}{\Phi(Q)}\Big) \;.
\end{equation}

To handle the odd primitive characters $\chi$, we use the Riemann
Hypothesis (Weil's theorem) to write
\begin{equation}
L(u,\chi) = \det (I-uq^{1/2}\Theta_\chi)\;,
\end{equation}
with the unitarized Frobenius $\Theta_\chi\in  U(\deg Q-1)$.
Hence for $2\leq n\leq k(\deg Q-1)$,
\begin{equation}\label{formula for Mdiv}
\EM(n;\divid_k\chi) =
(-1)^nq^{n/2}\sum_{\substack{j_1+\ldots+j_k=n\\ 0\leq j_1,\ldots,j_k
\leq \deg Q-1}} \Sc_{j_1} (\Theta_\chi )\cdot \ldots \cdot\Sc_{j_k} (\Theta_\chi) \;.
\end{equation}

Inserting \eqref{formula for Mdiv} into \eqref{var intermed odd} and
using \eqref{formula for Mdiv} and Katz's equidistribution theorem
\cite{KatzKR1} (here we require $Q$ squarefree) we  get for $\deg Q\geq 2$ and $2\leq n\leq k(\deg Q-1)$
\begin{equation}
\begin{split}
\lim_{q\to \infty} \frac{\Var(\ES_{\divid_k,n,Q})}{  q^n/|Q|} &=
\int_{U(\deg Q-1)} \Big| \sum_{\substack{j_1+\ldots +j_k=n\\0\leq
j_1,\ldots,j_k\leq \deg Q-1}}
 \Sc_{j_1}(U )\dots \Sc_{j_k}(U )\Big|^2 dU
\\&=I_k(n,\deg Q-1)\;,
\end{split}
\end{equation}
proving Theorem~\ref{thm var div AP}.

Note that If $n<\deg Q$, then we of course do not need these
powerful equidistribution results, since
there is at most  {\em one} $f$ with $\deg f=n$ and $f=A\bmod Q$,
which allows one to obtain the claim in an elementary manner.


\section{Matrix integral}\label{secRMT}

Our goal in this section is to evaluate the matrix integral \eqref{def of I}. We start by looking at the following products:
\begin{equation}
\begin{split}
\det (I-xU)^{k}\det (I-yU^{*})^{k}=
(\sum_{j=1}^{N}Sc_{j}(U)(-x)^{j})^{k}(\sum_{i=1}^{N}Sc_{i}(U^{*})(-y)^{i})^{k}
\end{split}
\end{equation}
We will be interested in the expected value over the unitary group of the above. Due to the invariance of Haar measure of $U(N)$ under multiplication by unit scalars, we are left with only the diagonal terms, i.e.
\begin{equation}\label{generating}
\int_{U(N)}\det (I-xU)^{k}\det (I-yU^{*})^{k}dU=\sum_{0\leq m\leq kN}I_{k}(m,N)(xy)^{m}
\end{equation}
This integral therefore serves as a generating series for the function $I_k(m;N)$.
Note that we may switch the sign of both $x$ and $y$ and retain the same right hand side.


\subsection{Evaluation in a certain range}\label{sec:RMT}
We now give the proof of Theorem~\ref{Thm: Ik(m,N) for large m}.
For the range $m\leq N$, we will apply the method of Diaconis-Gamburd \cite{DG}
to obtain
\begin{equation}\label{eval I first range}
I_k(m;N) = \binom{m+k^2-1}{k^2-1}, \quad m\leq N
\end{equation}
When $(k-1)N\leq m\leq kN$ we have a functional equation which allows us to compute the integral in this range.

\subsubsection{The functional equation}
\begin{lemma}\label{functional eq}
For $0 \leq m \leq kN$, the following functional equation holds,
\begin{equation}\label{Isim}
I_k(m;N) = I_k(kN-m;N).
\end{equation}
\end{lemma}
\begin{proof} \label{convert to general $k$}
We use the functional equation of  the characteristic polynomial of
a unitary matrix
\begin{equation}
\det(I+xU) = x^N \det(U)\det(I+x^{-1} U^{*})
\end{equation}
which implies that
\begin{equation}
\Sc_j(U) = \det(U) \Sc_{N-j}(U^{*}) = \det(U)
\overline{\Sc_{N-j}(U)}
\end{equation}
Therefore
\begin{equation*}
 \Sc_{j_{1}}(U) \cdots\Sc_{j_{k}}(U) \overline{\Sc_{l_{1}}(U)\cdots \Sc_{l_{k}}(U)}  =
\Sc_{N-l_{1}}(U) \cdots\Sc_{N-l_{k}}(U) \overline{\Sc_{N-j_{1}}(U)\cdots \Sc_{N-j_{k}}(U)}.
\end{equation*}
We change variables
\begin{equation}
m'=kN-m, \quad j_{i}'=N-j_{i},\quad l_{i}'=N-l_{i}
\end{equation}
and so obtain
\begin{equation}
\left| \sum_{\substack{j_1+\dots j_k=m\\
0\leq j_1,\dots, j_k \leq N}} \Sc_{j_1}(U) \dots \Sc_{j_k}(U)
\right|^2 = \left| \sum_{\substack{j'_1+\dots j'_k=m'\\
0\leq j'_1,\dots, j'_k \leq N}} \Sc_{j'_1}(U) \dots \Sc_{j'_k}(U)
\right|^2
\end{equation}
which implies \eqref{Isim}.
\end{proof}

\subsubsection{Review of Diaconis and Gamburd \cite{DG}}\label{dg}

Let $A=(a_{i,j})$ be an $m\times n$ matrix with non-negative integer
entries. Let $r_i = \sum_j a_{i,j}$ be the sum of the entries in the
$i$-th row, and $c_j = \sum_i a_{i,j}$ be the sum of the entries in
the $j$-th column. Set
\begin{equation}
\mbox{row}(A) = (r_1,\dots, r_m), \quad \mbox{col}(A) = (c_1,\dots,
c_n)
\end{equation}

Let $\lambda = (\lambda_1,\dots, \lambda_r) \in \mathbb N^r$ with
$\lambda_1\geq \lambda_2\geq \dots \geq \lambda_r$ be a partition of
$n$, so $n=\sum_i \lambda_i$. Denote by $m_i=m_i(\lambda)$ the
number of part of $\lambda$ equal to $i$, so an alternative notation
is
\begin{equation}
\lambda= \langle 1^{m_1} 2^{m_2} \dots \rangle
\end{equation}
Given two partitions $\mu=(\mu_1,\dots, \mu_m)$ and $\tilde \mu =
(\tilde \mu_1,\dots, \tilde \mu_n)$, denote by $N_{\mu,\tilde \mu}$
the number of $m\times n$ matrices $A$ with non-negative integer
entries so that $\mbox{row}(A)=\mu$ and $\mbox{col}(A)=\tilde \mu$.
For instance if $\mu  = (2,1,1) = \langle  1^2 2^1 \rangle$ and
$\tilde \mu =(3,1) = \langle 1^1 3^1 \rangle$ then $N_{\mu,\tilde
\mu} = 3$ with the corresponding matrices $A$ being
$$ \begin{pmatrix} 2& 0\\1&0\\0&1\end{pmatrix},\quad
 \begin{pmatrix} 2& 0\\0&1\\1&0\end{pmatrix},\quad
 \begin{pmatrix} 1& 1\\1&0\\1&0\end{pmatrix}
$$


We quote a result of Diaconis and Gamburd:
\begin{theorem}\cite{DG}\label{thm:DG}
 Let   $a_i,b_j$ be non-negative integers, $1\leq i,j \leq \ell$.
Then for $  \max(\sum_{j=1}^\ell ja_j,\sum_{j=1}^\ell jb_j)\leq N$,
\begin{equation}\label{eq:DGb}
\int_{U(N)} \prod_{j=1}^\ell (\Sc_j(U))^{a_j}
\overline{\Sc_j(U)}^{b_j} dU = N_{\mu,\tilde \mu}
\end{equation}
where $\mu = \langle 1^{a_1} 2^{a_2}\dots \ell^{a_\ell}\rangle$,
$\tilde \mu = \langle 1^{b_1} 2^{b_2}\dots \ell^{b_\ell}\rangle$.
\end{theorem}

\subsubsection{Back to the variance calculation}
There is a slight reformulation of Theorem \ref{thm:DG} that will be useful to have stated. Let $\mu = (j_1,...,j_k)$ and $\tilde{\mu} = (\tilde{j_1},...,\tilde{j_k})$ be arrays of non-negative integers (we now impose no condition that they be weakly decreasing), and we generalize $N_{\mu,\tilde{\mu}}$ in the obvious manner, so that it is the count of $k\times k$ matrices $A$ with non-negative integer entries such that $\mbox{row}(A) = \mu$ and $\mbox{col}(A) = \tilde{\mu}$. Note that, by permuting rows and then columns of the matrix $A$, if the arrays $\mu$ and $\nu$ are rearrangements of each other, and likewise for $\tilde{\mu}$ and $\tilde{\nu}$,
$$
N_{\mu,\tilde{\mu}} = N_{\nu,\tilde{\nu}}.
$$
Thus Theorem \ref{thm:DG} may be reformulated as the statement that for $\max(\sum j_i, \sum \tilde{j_i}) \leq N$,
\begin{equation}
\label{eq:DGb2}
\int_{U(N)} \prod_i \Sc_{j_i}(U) \overline{\Sc_{\tilde{j}_i}(U)}\,dU = N_{\mu,\tilde{\mu}}.
\end{equation}
The reformulation is useful for us because in the proof that follows we will be working with arrays that are not ordered.
 


\begin{proof}[Proof of Theorem \ref{Thm: Ik(m,N) for large m}]
For $m\leq N$, note that in the definition \eqref{def of I}, the restriction that $j_i \leq N$ plays no role. Hence,
$$
I_k(m;N):=\int_{U(N)} \left| \sum_{j_1+\cdots +j_k=m} \Sc_{j_1}(U) \dots \Sc_{j_k}(U) \right|^2 dU.
$$
We may expand the square, and, because in the range of summation over $j_i$ we have $j_1+\cdots+j_k=m\leq N$, we may apply \eqref{eq:DGb2} to see that the above expression is just
$$
\sum_{\substack{j_1+\cdots +j_k=m \\ \,\tilde{j}_1+\cdots+ \tilde{j}_k = m}} N_{\mu,\tilde{\mu}}.
$$
But this sum is just the count of all $k\times k$ matrices comprised of non-negative integer entries with the total sum of the entries being $m$. This in turn is just the number of ways of writing $a_1+\cdots+a_{k^2} = m.$ Therefore, for this range of $m\leq N$, $I_k(m;N)$ is the binomial coefficient
$$
I_k(m;N) = \binom{m+k^{2}-1}{k^{2}-1}.
$$
One way to see so is to note that it is the coefficient of $x^m$ in
$$
\sum_{a_i\geq 0} x^{a_1+\cdots+a_{k^2}} = \frac{1}{(1-x)^{k^2}}.
$$

Finally, to deal with the case $(k-1)N \leq m \leq kN$, we use the functional equation, Lemma \ref{functional eq}.
\end{proof}

\subsection{Evaluation in other ranges}\label{sec:other ranges}
It was shown in the previous section how to evaluate $I_k(m;N)$ in the ranges $m\leq N$ and $(k-1)N \leq m \leq kN$.  Our goal here is to illustrate a general method for computing it in all other ranges.  

By \eqref{generating},  we are looking to find the coefficient of $x^{m}$ in the expansion of
\begin{equation}\label{P_k def}
P_k(x)=\int_{U(N)} \det(I-U^{\ast}x)^k\det(I-U)^kdU.
\end{equation}    
This can be calculated using the following Theorem:

\begin{theorem}\cite{cfkrs1},\cite{cfkrs2}\label{partial_sums}
Let $A$ and $B$ be finite collections of complex numbers. Then
\begin{align}
\notag \int_{U(N)} \prod_{\alpha\in A}\det(I-U^{\ast}e^{-\alpha})&
\prod_{\beta\in B}\det(I-Ue^{-\beta})\,dU 
\\ &  = \sum_{\substack{S\subseteq A\\T\subseteq
B\\|T|=|S|}}e^{-N(\sum_{\hat{\alpha}\in s}
\hat{\alpha}+\sum_{\hat{\beta}\in
T}\hat{\beta})}Z(\overline{S}+T^{-},\overline{T}+S^{-})
\end{align}
Where
$$\overline{S}=A-S,~~~~\overline{T}=B-T,~~~~S^{-}=\{-\hat{\alpha},\hat{\alpha}\in
S\},~~~~T^{-}=\{-\hat{\beta},\hat{\beta}\in T\}$$ and
$$Z(A,B)=\prod_{\substack{\alpha\in A\\\beta\in B}}z(\alpha+\beta)$$
with $z(x)=\frac{1}{1-e^{-x}}.$
\end{theorem}

For example, we find 
\begin{equation}
\begin{split}
P_{2}(x)=\frac{1}{(1-x)^{4}}[&1+x^{2N+4}-(2+N)^{2}x^{1+N}+
 \\&2(3+4N+N^{2})x^{N+2}-(2+N)^{2}x^{N+3}]
\end{split}
\end{equation}
Note that $P_{2}(x)$ satisfies $x^{2N}P_{2}(x)(1/x)=P_{2}(x)(x)$, which corresponds to the functional equation $I_{2}(m;N)=I_{2}(2N-m;N)$.  Evaluating the coefficient of $x^m$ we recover
\begin{equation}
 I_{2}(m;N)=
\begin{cases}
\binom{m+3}{3}  &\mbox{if } m \leq N \\
\binom{2N-m+3}{3} & \mbox{if }N\leq m\leq 2N .
\end{cases}\end{equation}
as proved in the previous section.

Similarly 
\begin{multline}\nonumber
 P_{3}(x)=\frac{1}{(1-x)^{9}}[1-x^{3N+9}+
(3+N)^{2}(4+5N+N^{2})(x^{2+N}-x^{N+7})+
\\ \frac{1}{4}(3+N)^{2}(2+N)^{2}(x^{2N+8}-x^{N+1})+
(3+N)^{2}(10+7N+N^{2})(x^{N+4}-x^{2N+5})+
\\(3+N)^{2}(N+4)^{2}(1/4)(x^{2N+4}-x^{N+5})+
\frac{3}{2}(56+90N+51N^{2}+12N^{3}+N^{4})(x^{2N+6}-x^{N+3})]
\end{multline}
Again $P_{3}(x)$ satisfies $x^{3n}P_{3}(x)(1/x)=P_{3}(x)(x)$, corresponding to the functional equation $I_{3}(m;N)=I_{3}(3N-m;N)$. Hence
\begin{equation}
 I_{3}(m;N)=
\begin{cases}
\binom{m+8}{8}  &\mbox{if } m \leq N+3 \\
\Poly_{8}(m)&\mbox{if } N+3<m < 2N-3 \\
\binom{3N-m+8}{8} & \mbox{if } 2N-3\leq m\leq 3N .
\end{cases}\end{equation}
Where $\Poly_{8}(m)$ is a polynomial in $m$ of degree $8$, and is given by
\begin{multline}\nonumber
\Poly_{8}(m)=\binom{m+8}{8}-
(3+N)^{2}(N+4)^{2}(1/4)\binom{m-N+3}{8}+\\
(3+N)^{2}(10+7N+N^{2})\binom{m-N+4}{8}-
\frac{3}{2}(56+90N+51N^{2}+12N^{3}+N^{4})\binom{m-N+5}{8}+\\
(3+N)^{2}(4+5N+N^{2})\binom{m-N+6}{8}-
\frac{1}{4}(3+N)^{2}(2+N)^{2}\binom{m-N+7}{8}
\end{multline}

This method obviously extends to larger values of $k$, but in practice is effective when $k$ is relatively small.

\subsection{Large N asymptotics: A symmetric function theory approach}\label{asymptotic1}

In this subsection, we give a proof of Theorem \ref{thm:Asymp of I}, determining the asymptotic behavior of $I_k(m; N)$ when $m$ and $N$ grow in ratio to one another. We begin however with a proof of Theorem \ref{lattice_count}, the characterization of $I_k(m;N)$ in terms of a count of lattice points.
It is then in part by estimating this lattice count that we obtain the coefficient $\gamma_k(c)$ in Theorem \ref{thm:Asymp of I}.

\subsubsection{Some preliminaries from symmetric function theory}

The proof below of Theorems \ref{lattice_count} and \ref{thm:Asymp of I} requires some knowledge from symmetric function theory. In order to make our presentation self-contained, in this section we recall for the reader a few concepts that will be necessary. In particular Schur functions, defined below, will play a key role. The reader already familiar with this material may skip ahead to the next subsection. (Standard references for this material include \cite{Bu, Ga, St}; for readers with a background in analytic number theory, \cite{Ga} is perhaps the quickest general introduction.)

Recall (from \ref{dg}), a \emph{partition} $\lambda$ is a sequence $(\lambda_1,...,\lambda_k)$ of positive integers satisfying $\lambda_1 \geq \lambda_2 \geq \dots \geq \lambda_k$. The \emph{length} $\ell(\lambda)$ of such a partition is defined by $\ell(\lambda):=k$. If $1$ appears among the numbers $\lambda_1,...,\lambda_k$ a total of $m_1$ times, $2$ appears $m_2$ times, and so on, we also write $\lambda = \langle 1^{m_1}2^{m_2}\cdots \rangle.$

A \emph{Young diagram} is a collection of boxes arranged in left-justified rows, with a weakly decreasing number of boxes in each row. The partition $(\lambda_1,...,\lambda_k)$ corresponds to a Young diagram with $\lambda_1$ boxes in the first row, $\lambda_2$ boxes in the second, and so on to $\lambda_k$ boxes in the $k$th row. For instance, the partition $(6,4,3,1)$ corresponds to the Young diagram
$$
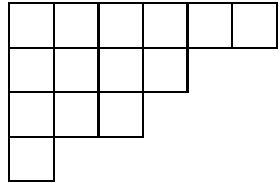
\ytableausetup{mathmode, boxsize=1.5em}
\ytableaushort{\none,\none,\none,\none}*{6,4,3,1}
$$

For $\lambda$ a partition, a \emph{semistandard Young tableau (SSYT)} of shape $\lambda$ is an array $T = (T_{ij})_{1\leq i \leq \ell(\lambda), 1\leq j \leq \lambda_i}$ of positive integers such that $T_{i,j} \leq T_{i,j+1}$ and $T_{ij} < T_{i+1,j}$. It is common to write SSYTs in a Young diagram, as for example
$$
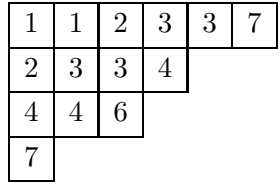
\ytableausetup{mathmode, boxsize=1.5em}
\ytableaushort{112337,2334,446,7}
$$
This is a SSYT of shape $(6,4,3,1)$. Note that the condition $T_{i,j} \leq T_{i,j+1}$ translates to the array $T$ weakly increasing in every row and $T_{i,j} < T_{i+1,j}$ to strictly increasing in every column.

$T$ has \emph{type} $a = (a_1,a_2,...)$ if $T$ has $a_i =  a_i(T)$ parts equal to $i$. The SSYT above has type $(2,2,4,3,0,1,2)$. It is common to use the notational abbreviation
$$
x^T = x_1^{a_1(T)}x_2^{a_2(T)}\cdots,
$$
so for the example SSYT above,
$$
x^T = x_1^2 x_2^2 x_3^4 x_4^3 x_6 x_7.
$$
We finally come to the combinatorial definition of Schur functions.
\begin{defn}
\label{schur_combinatorial}
For a partition $\lambda$, the Schur function in the variables $x_1,...,x_r$ indexed by $\lambda$ is a multivariable polynomial defined by
$$
s_\lambda(x_1,...,x_r) := \sum_T x_1^{a_1(T)}\cdots x_r^{a_r(T)},
$$
where the sum is over all SSYTs $T$ whose entries belong to the set $\{1,...,r\}$ (i.e. $a_i(T) = 0$ for $i > r$).
\end{defn}

For example, the SSYTs of shape $(2,1)$ whose entries belong to the set $\{1,2,3\}$ are
\vspace{2mm}

\ytableausetup{mathmode, boxsize=1.1em}
\ytableaushort{11,2}
\hspace{3mm}
\ytableausetup{nobaseline}
\ytableaushort{12,2}
\hspace{3mm}
\ytableausetup{nobaseline}
\ytableaushort{13,2}
\hspace{3mm}
\ytableausetup{nobaseline}
\ytableaushort{11,3}
\hspace{3mm}
\ytableausetup{nobaseline}
\ytableaushort{12,3}
\hspace{3mm}
\ytableausetup{nobaseline}
\ytableaushort{13,3}
\hspace{3mm}
\ytableausetup{nobaseline}
\ytableaushort{22,3}
\hspace{3mm}
\ytableausetup{nobaseline}
\ytableaushort{23,3}

\vspace{4mm}

\noindent and so
$$
s_{(2,1)}(x_1,x_2, x_3) = x_1^2 x_2 + x_1 x_2^2 + x_1^2 x_3 + x_1 x_3^2 + x_2^2 x_3 + x_2 x_3^2 + 2 x_1 x_2 x_3.
$$

\subsubsection{A proof of Theorems \ref{lattice_count} and \ref{thm:Asymp of I}}




\begin{proof}[Proof of Theorem \ref{lattice_count}]
Our starting point is again equation \eqref{generating}, which in this case we evaluate using a result of Bump and Gamburd \cite[Prop. 4]{BuGa}: 
\begin{theorem}\cite{BuGa}
\label{bump_gamburd_thm}
Let $\alpha_1,...,\alpha_{L+L'}$ be complex numbers. Then,
$$
\int_{U(N)} \prod_{\ell=1}^L \det(1+ \alpha_\ell^{-1} U^{-1}) \prod_{\ell'=1}^{L'}\det(1+ \alpha_{L+\ell'} U)\,dU = \frac{s_{\langle N^L\rangle}(\alpha_1,...,\alpha_{L+L'})}{\alpha_1^N\cdots \alpha_L^N}.
$$
\end{theorem}
Here $s_{\langle N^L\rangle }$ is a Schur function indexed by the partition $\langle N^L \rangle$.

By specializing this Theorem, we see that,
\begin{equation}
\label{bump_gamburd}
\int_{U(N)} \det(1+\alpha U)^k \det(1+\beta U^{-1})^k\, dU = \alpha^{kN} s_{\langle N^k\rangle}(\underbrace{\alpha^{-1},...,\alpha^{-1}}_{k\textrm{ terms}}, \underbrace{\beta, ..., \beta}_{k\textrm{ terms}}).
\end{equation}
Expanding the Schur function as a polynomial and labeling the coefficients, we have
$$
\alpha^{kN} s_{(N^k)}(\underbrace{\alpha^{-1},...,\alpha^{-1}}_{k\textrm{ terms}}, \underbrace{\beta, ..., \beta}_{k\textrm{ terms}}) = \sum c_{ij} \alpha^i \beta^j.
$$
By comparison with \eqref{generating}, we see that $I_k(m;N) = c_{mm}.$

From the combinatorial definition of Schur functions (Definition \ref{schur_combinatorial} above), we see that $c_{mm}$ is the number of semistandard Young tableau (SSYT) $T$ such that if, as before, $a_i$ denotes the number of $i$'s in $T$,
$$
a_{k+1}+\cdots+ a_{2k} = m,
$$
and $a_i = 0$ for $i > 2k$.

We parametrize such tableaux $T$ by letting $y^{(s)}_r = y^{(s)}_r(T)$ be the rightmost position of the entry $s$ in row $r$; if $s$ does not occur in row $r$, inductively define $y^{(s)}_r = y^{(s-1)}_r$, with $y_r^{(1)} = 0$ if the entry $1$ does not occur in row $r$. So, for instance, in the SSYT $T$ on the partition $(7^2)$ with entries ranging from $1$ to $4$ given by 
$$
T = \ytableausetup{nobaseline, boxsize = 1.5em}
\ytableaushort{1112333,2224444}
$$
we have
$$
\begin{pmatrix}
y_1^{(1)} & y_1^{(2)} & y_1^{(3)} & y_1^{(4)} \\
y_2^{(1)} & y_2^{(2)} & y_2^{(3)} & y_2^{(4)} 
\end{pmatrix}
=
\begin{pmatrix}
3 & 4 & 7 & 7 \\
0 & 3 & 3 & 7
\end{pmatrix}.
$$ 
Note that here $y_2^{(1)}=0$ and $y_1^{(3)} = y_1^{(4)} = y_2^{(4)} = 7.$ That these entries should take these values is necessarily the case; if the $2^{\textrm{nd}}$ row began with $1$, then $T$ could not be made to be strictly increasing in columns, and for the same reason the $1^{\textrm{st}}$ row may not end with $4$. 

Moreover, note that because rows increase weakly,
\begin{equation}
\label{order1}
y_r^{(s)} \leq y_r^{(s+1)}
\end{equation}
and because columns increase strongly,
\begin{equation}
\label{order2}
y_{r+1}^{(s+1)} \leq y_r^{(s)}.
\end{equation}
With these restrictions \eqref{order1} and \eqref{order2} in place, there is a bijection between arrays
$$
\begin{pmatrix}
y_1^{(1)} & y_1^{(2)} & \cdots & y_1^{(k)} & N           & \cdots & \cdots & N \\
0         & y_2^{(2)} & \cdots & y_2^{(k)} & y_2^{(k+1)} & N      & \cdots & N \\
\vdots    & \ddots    & \ddots & \ddots    & \ddots      & \ddots &
\ddots & \vdots \\
0         & 0         & \cdots & y_k^{(k)} & y_k^{(k+1)} & \cdots &  y_k^{(2k-1)} & N
\end{pmatrix}
$$
with $y_r^{(s)} \in [0,N] \cap \mathbb{Z}$ and SSYT of $\langle N^k\rangle$ with entries ranging from $1$ to $2k$.

It is easy to see that those SSYT for which $a_{k+1} + \cdots + a_{2k}=m$ correspond to those arrays in which $(N-y_1^{(k)}) + (N- y_2^{(k)}) + \cdots + (N- y_k^{(k)}) = m.$ By re-indexing $x_r^{(s)} = y_r^{(s+r-1)}$, we obtain the proposition.
\end{proof}

With Theorem \ref{lattice_count} in hand, getting an expression for $\gamma_k(c)$ in Theorem \ref{thm:Asymp of I}, as we will see, is a more or less standard argument in counting lattice points. On the other hand, in order to simplify the expression we get to \eqref{def of gamma2}, it will be useful to have done the following computation beforehand.

\begin{lemma}
\label{Vand_Reduce}
As usual, define the Vandermonde determinant by 
$$
\Delta(w_1,w_2,...,w_k) := \prod_{i > j} (w_i - w_j),
$$
and for $\beta \in \mathbb{R}^{k+1}$ satisfying $\beta_1 \leq \beta_2 \leq \cdots \leq \beta_{k+1}$, define
$$
I(\beta) = \{\alpha\in \mathbb{R}^k: \beta_1 \leq \alpha_1 \leq \beta_2 \leq \alpha_2 \leq \cdots \leq \alpha_k \leq \beta_{k+1}\}.
$$
Then 
\begin{equation}
\label{vand_reduce}
\int_{\alpha\in I(\beta)} \Delta(\alpha_1, \alpha_2, ..., \alpha_k) \, d^k\alpha = \frac{1}{k!} \Delta(\beta_1,...,\beta_{k+1}).
\end{equation}
\end{lemma}

\begin{proof}
Because of the well known identity $\Delta(w) = \det(w_\mu^{\nu-1})$, we see that the left hand side of \eqref{vand_reduce} is just
\begin{align*}
&\int_{\alpha\in I(\beta)} \det\begin{pmatrix} 1 & 1 & \cdots & 1 \\
\alpha_1 & \alpha_2 & \cdots & \alpha_{k} \\
\vdots & \vdots & \ddots & \vdots \\
\alpha_1^{k-1} & \alpha_2^{k-1} & \cdots & \alpha_{k}^{k-1} \end{pmatrix}\, d^{k}\alpha \\
&= \det\begin{pmatrix} \beta_2-\beta_1 & \beta_3-\beta_2 & \cdots & \beta_{k+1}-\beta_k \\
(\beta_2^2-\beta_1^2)/2 & (\beta_3^2-\beta_3^2)/2 & \cdots & (\beta_{k+1}^2-\beta_k^2)/2 \\
\vdots & \vdots & \ddots & \vdots \\
(\beta_2^k-\beta_1^k)/k & (\beta_3^k-\beta_2^k)/k & \cdots & (\beta_{k+1}^k-\beta_k^k)/k \end{pmatrix}
\end{align*}
by integrating one variable at a time and using multilinearity. But again applying multilinearity (twice), we see that this is just
$$
\frac{1}{k!} \sum_{\varepsilon\in \{0,1\}^k} (-1)^{k-|\varepsilon|} \det\begin{pmatrix} 
\beta_{1+\varepsilon_1} & \beta_{2+\varepsilon_2} & \cdots & \beta_{k+\varepsilon_k} \\
\beta_{1+\varepsilon_1}^2 & \beta_{2+\varepsilon_2}^2 & \cdots & \beta_{k+\varepsilon_k}^2 \\
\vdots & \vdots & \ddots & \vdots \\
\beta_{1+\varepsilon_1}^k & \beta_{2+\varepsilon_2}^k & \cdots & \beta_{k+\varepsilon_k}^k 
\end{pmatrix},
$$
where $|\varepsilon|$ is the number of $i$ such that $\varepsilon_i=1$. Clearly the determinant in the summand will be $0$ unless $\varepsilon$ is one of the $k+1$ possibilities: $(1,1,1,...,1)$, $(0,1,1,...,1)$, $(0,0,1,...,1)$, ..., $(0,0,0,...,0)$. Thus the sum above is just a Laplace expansion of
$$
\frac{1}{k!} \det\begin{pmatrix} 1 & 1 & \cdots & 1 \\
\beta_1 & \beta_2 & \cdots & \beta_{k+1} \\
\vdots & \vdots & \ddots & \vdots \\
\beta_1^k & \beta_2^k & \cdots & \beta_{k+1}^k \end{pmatrix}
$$
as claimed.
\end{proof}

\begin{proof}[Proof of Theorem \ref{thm:Asymp of I}]
We demonstrate first that \eqref{thm:poly_approx} of Theorem \ref{thm:Asymp of I} holds with $\gamma_k(c)$ given by
\begin{equation}
\label{gamma1}
\gamma_k(c) = \int_{[0,1]^{k^2}} \delta_c(u_1^{(k)} + u_2^{(k-1)} + \cdots + u_k^{(1)}) \mathbf{1}_{A_k}(u)\, d^{k^2}u,
\end{equation}
where $\mathbf{1}_{A_k}$ is the indicator function of the set ${A_k}$ (defined in the statement of Theorem \ref{lattice_count}). 

The truth of this should come as no surprise; we have just approximated a lattice count with a continuous approximation. Later we show that this integral is equal to the right hand side of \eqref{def of gamma2}.

Our proof of this first part is standard. For notational reasons let $S = \{(i,j): 1\leq i,j \leq k: (i,j) \neq (1,k)\}$, and let $V_c$ be the convex region contained in $\mathbb{R}^{k^2-1}= \{(u_i^{(j)})_{(i,j)\in S}: u_i^{(j)}\in \bb{R}\}$ defined by the following system of inequalities:
\begin{enumerate}
\item $0 \leq u_i^{(j)} \leq 1,$ for all  $(i,j) \in S,$
\item For $u_1^{(k)} := c - (u_2^{(k-1)} + \cdots + u_k^{(1)})$, we have $0\leq u_1^{(k)} \leq 1$, and
\item The matrix $(u_i^{(j)})_{1\leq i,j\leq k}$ lies in the set $A_k$.

\end{enumerate}


This region is convex because it is the intersection of half planes. Note moreover that for all $c\in [0,k]$, the region $V_c$ is contained in $[0,1]^{k^2-1}$, and therefore contained in a closed ball of radius $\sqrt{k^2-1}$.

Theorems \ref{lattice_count} and Lemma \ref{functional eq} show that
\begin{equation}
\label{dilation_count}
I_k(m;N) = \# (\mathbb{Z}^{k^2-1} \cap (N\cdot V_c)),
\end{equation}
where $N\cdot V_c = \{Nx: x \in V_c\}$ is the dilate of $V_c$ by a factor of $N$. 

We will need to reference the well known principle that a count of lattice points in a region can be approximated by the volume of the region (at least in ordinary circumstances). A result of the sort we quote below dates back to Davenport \cite{Da1,Da2}; the clean formulation we have cited here may be found in \cite[Section 2]{Sc}.
\begin{theorem}
\label{lattice_to_volume}
If $S \subset \mathbb{R}^\ell$ is a convex region contained in a closed ball of radius $\rho$, then
\begin{equation}
\label{schmidt}
\#(S\cap \mathbb{Z}^\ell) = \mathrm{vol}_\ell(S) + O(\rho^{\ell-1}),
\end{equation}
where the implicit constant depends only on $\ell$.
\end{theorem} 

Applying \eqref{schmidt}, with $\ell=k^2-1$, we see
$$
I_k(m;N) = \mathrm{vol}(N\cdot V_c) + O_k(N^{k^2-2}).
$$
Yet clearly
$$
\mathrm{vol}(N\cdot V_c) = N^{k^2-1} \int_{[0,1]^{k^2}} \delta_c(u_1^{(k)} + \cdots + u_k^{(1)}) \mathbf{1}_{A_k}(u)\, d^{k^2} u,
$$
which implies \eqref{thm:poly_approx}, with $\gamma_k(c)$ given by \eqref{gamma1}.

It remains to show that this integral can be reduced to the expression defined in \eqref{def of gamma2}. Here we make use of Lemma \ref{Vand_Reduce}. We have, by applying it inductively,
\begin{align*}
\gamma_k(c) &= \int_{[0,1]^{k^2}} \delta_c(u_1^{(k)} + \cdots + u_k^{(k)})\cdot \mathbf{1}\left[\begin{smallmatrix} 
u_1^{(1)} & \leq & u_1^{(2)}  & \leq & \cdots &      & \cdots          &      & \cdots        \\
\vleq     &      & \vleq      &      &        &      &                 &      &               \\
u_2^{(1)} & \leq & u_2^{(2)}  & \leq & \cdots &      & \cdots          &      & \cdots        \\
\vleq     &      & \vleq      &      &        &      &                 &      &               \\
\vdots    &      & \vdots     &      & \ddots &      & \vdots          &      & \vdots        \\
          &      &            &      &        &      & \vleq           &      & \vleq         \\
\cdots    &      & \cdots     &      & \cdots & \leq & u_{k-1}^{(k-1)} & \leq & u_{k-1}^{(k)} \\
          &      &            &      &        &      & \vleq           &      & \vleq         \\
\cdots    &      & \cdots     &      & \cdots & \leq & u_k^{(k-1)}     & \leq & u_k^{(k)} 
\end{smallmatrix}\right]
\, d^{k^2} u \\
&= \int_{[0,1]^{k^2-2}}\delta_c(u_1^{(k)} + \cdots + u_k^{(1)})\cdot
\mathbf{1}\left[\begin{smallmatrix} 
          &      & u_1^{(2)}  & \leq & \cdots &      & \cdots          &      & \cdots        \\
          &      & \vleq      &      &        &      &                 &      &               \\
u_2^{(1)} & \leq & u_2^{(2)}  & \leq & \cdots &      & \cdots          &      & \cdots        \\
\vleq     &      & \vleq      &      &        &      &                 &      &               \\
\vdots    &      & \vdots     &      & \ddots &      & \vdots          &      & \vdots        \\
          &      &            &      &        &      & \vleq           &      & \vleq         \\
\cdots    &      & \cdots     &      & \cdots & \leq & u_{k-1}^{(k-1)} & \leq & u_{k-1}^{(k)} \\
          &      &            &      &        &      & \vleq           &      &               \\
\cdots    &      & \cdots     &      & \cdots & \leq & u_k^{(k-1)}     &      & 
\end{smallmatrix}\right]\\
&\hspace{47mm} \times \frac{\Delta(u_1^{(2)}, u_2^{(1)})}{1!} \frac{\Delta(u_{k-1}^{(k)}, u_{k}^{(k-1)})}{1!}
\, d^{k^2-2} u \\
&=\cdots \\
&= \int_{[0,1]^k} \delta_c(u_1^{(k)} + \cdots + u_k^{(1)})\cdot \mathbf{1}(u_k^{(1)} \leq u_{k-1}^{(2)} \leq \cdots \leq u_1^{(k)})\\
& \hspace{22mm} \times \frac{\Delta(u_1^{(k)}, u_{2}^{(k-1)}, \cdots, u_k^{(1)})}{1!\cdot 2! \cdots (k-1)!}\cdot\frac{\Delta(u_1^{(k)}, u_{2}^{(k-1)}, \cdots, u_k^{(1)})}{1!\cdot 2! \cdots (k-1)!} \,d^k u \\
&= \frac{1}{k!\, G(1+k)^2} \int_{[0,1]^k}  \delta_c(w_1+\cdots+ w_k) \Delta(w)^2\, d^k w,
\end{align*}
with the last step following from symmetry.
\end{proof}

We note for the reader familiar with Gelfand-Tsetlin patterns that what we have done in these last few steps is to compute the volume of what is called a Gelfand-Tsetlin polytope. A computation of this volume has appeared before in the literature (see \cite{Ba} for a proof using representation theory, or \cite{Ol} for a proof using the  Harish-Chandra-Itzykson-Zuber integral), but the elementary proof we give here based on Lemma \ref{Vand_Reduce} seems to be new.

\subsubsection{Ehrhart theory}
Theorem \ref{lattice_count} also allows us to say something about the algebraic character of the quantities we have been discussing.

\begin{cor}
\label{ehrhart}
Let $c=p/q$ be fixed rational number and $k$ be a fixed integer. If $N$ is a multiple of $q$, then $I_k(cN, N) = P_{c,k}(N)$, where $P_{c,k}$ is a polynomial of degree $k^2-1$.
\end{cor}
\begin{proof}
This corollary follows from an application of a theorem of Ehrhart \cite{Er}: 
\begin{theorem}
\label{er_theorem}
If $E$ is a convex lattice polytope in $\mathbb{R}^n$ (that is, a polytope whose vertices are all integer coordinates), then there is a polynomial $P$ of degree $n$, such that for all $\ell \in \mathbb{N}_{> 0}$,
$$
\#(\mathbb{Z}^n \cap (\ell\cdot E)) = P(\ell).
$$
\end{theorem}
Returning to the corollary at hand, we have from \eqref{dilation_count}, when $N=q\ell$,
$$
I_k(cN;N) = \# (\mathbb{Z}^{k^2-1} \cap (\ell\cdot[q\cdot V_c])).
$$
But then it is straightforward to verify that $q V_c = q V_{p/q}$ is a convex lattice polytope in $\mathbb{R}^{k^2-1}$, so that $I_k(cN;N)$ is a polynomial in $\ell$ and therefore in $N$.
\end{proof}

\subsection{Large N asymptotic: the complex analysis approach}\label{asymptotic2}

In this subsection we prove Theorem~\ref{integral analytic}.  The approach we take is based on the following expression proved in \cite{cfkrs1} (Lemma 2.1):

\begin{theorem}\cite{cfkrs1} \label{complex integral representation}
Let $\alpha_{i},\beta_{j}$ be complex numbers. Then,
\begin{equation*}
\begin{split}
\int_{U(N)}& \prod_{i=1}^{r}\det(I-U^{\ast}e^{-\alpha_{i}})
\prod_{j=1}^{r}\det(I-Ue^{-\beta_{j}})dU\\
&=\frac{(-1)^{r}e^{N(\alpha_{1}+\cdots+\alpha_{r})}}{(2\pi i)^{2r}(r!)^{2}}\oint\cdots\oint e^{-N(z_{r+1}+\cdots+z_{2r})}\prod_{\substack{1\leq l\leq r\\r+1\leq q \leq 2r}}(1-e^{z_{q}-z_{l}})^{-1}
\\&\times
\frac{\Delta(z_{1},\ldots,z_{2r})^{2}}
{\prod_{i=1}^{2r}\prod_{j=1}^{r}(z_{i}-\alpha_{j})(z_{i}-\beta_{j})}
dz_{1}\cdots dz_{2r}
\end{split}
\end{equation*}
Where $\Delta(z_1,\dots, z_{2r}) = \prod_{i<j}(z_j-z_i)$ is the vandermonde determinant, and the contour integrals enclose the variables $\alpha_{i},\beta_{j}.$

\end{theorem}

From the definition \eqref{P_k def}, we have
\begin{equation}
\begin{split}
P_{k}(x)=& \\& \frac{(-1)^{k}x^{kN}}{(2\pi i)^{2k}(k!)^{2}}
\oint\cdots\oint e^{-N(z_{k+1}+\cdots+z_{2k})}\prod_{\substack{1\leq l\leq k\\k+1\leq q \leq 2k}}(1-e^{z_{q}-z_{l}})^{-1}
\\&\times
\frac{\Delta(z_{1},\ldots,z_{2k})^{2}}
{\prod_{i=1}^{2k}((z_{i}-a)z_{i})^{k}}
dz_{1}\cdots dz_{2k}
\end{split}
\end{equation}
where $a=\log x$.


We set $m=cN$ and will consider when $0\leq c\leq k $ is fixed and $N\to \infty$ in such a way that $cN$ is an integer. We will then need to compute the coefficient of $x^{cN}$ in $P_{k}(x)$.

We first shrink the contour in \eqref{complex integral representation} into small circles centered at $0$ and $a$.  This leads to
a sum of $2^{2k}$ multiple integrals, each surrounding either $0$ or $a$; c.f.~the calculation in \cite{KO}.  Taking into account symmetries between the variables and counting the number of ways of picking $\ell$  of the first $k$ contours to surround $a$, and $k-\ell$ of the second $k$ contours to surround $a$, we find
\begin{equation}
P_k(x) = 
\sum_{\ell=0}^k \binom{k}{\ell}^2 P_{k,\ell}(x)
\end{equation}
where $P_{k,\ell}(x)$ is the integral with contours
$z_1,\dots, z_\ell$, $ z_{k+\ell+1}$, $\dots ,z_{2k}$ along small
circles surrounding $a=\log x$ and 
$z_{\ell+1},\dots, z_{k+\ell}$ along small circles surrounding $0$.
The remaining integrals where there are different numbers of
contours surrounding $a$ and $0$ do not contribute, as proved in the following Lemma \ref{lem: zero terms}, which we prove below.

Next we change variables
$$
z_j = \epsilon_j a  + \frac{v_j}N
$$
where
$$
\epsilon_j = \begin{cases} 1,& j=1,\ldots, \ell \; {\rm or}\;  j=k+\ell+1,\ldots, 2k \\
0,& \ell+1\leq j\leq k+\ell \end{cases}
$$
This gives that the integrand of $P_{k,\ell}(x)$ is, up to terms of order $1/N$ smaller,
\begin{equation}\label{1.4}
\frac{x^{-N(k-\ell)}}{N^{2k}}
 \frac { e^{-(v_{k+1}+\ldots +v_{2k})}
 \prod\limits_{\substack{i<j\\ \epsilon_i\neq \epsilon_j}}a^2 \prod\limits_{\substack{i<j\\\epsilon_i=\epsilon_j}} (\frac{v_i-v_j}{N})^2 dv_1\dots dv_{2k}}
{  \prod\limits_{\substack{t \leq k<q\\ \epsilon_t = \epsilon_q}}
\frac{v_q-v_t}{N} \prod\limits_{\substack{t \leq k<q\\
\epsilon_t\neq \epsilon_q }}
 (1-x^{\epsilon_q-\epsilon_t}e^{\frac{v_q-v_t}{N}} )     a^{2k^2}  (-1)^{k^2} \prod\limits_{j=1}^{2k} (\frac{v_j}{N} )^{k}  } \;.
\end{equation}

The number of   pairs  $i<j$ with $\epsilon_i\neq \epsilon_j$ is   $k^2$, hence $\prod_{\substack{i<j\\ \epsilon_i\neq \epsilon_j}}a^2 = a^{2k^2}$;
and the number of   pairs  $i<j$ with $\epsilon_i = \epsilon_j$ is   $\binom{2k}{2}-k^2 =k^2-k$, so that
$$
\prod_{\substack{i<j\\\epsilon_i=\epsilon_j}} (\frac{v_i-v_j}{N})^2 =
\frac 1{N^{2(k^2-k)}} \prod_{\substack{i<j\\\epsilon_i=\epsilon_j}} (v_i-v_j)^2 \;.
$$
The number of pairs $(t,q)$ with $1\leq t\leq k<q\leq 2k$ and $\epsilon_t=\epsilon_q$ is $2\ell(k-\ell)$, hence
$$
\prod_{\substack{t \leq k<q\\ \epsilon_t = \epsilon_q}} \frac{v_q-v_t}{N} =
\frac 1{N^{2\ell(k-\ell)}}  \prod_{\substack{t \leq k<q\\ \epsilon_t = \epsilon_q}} (v_q-v_t) \;.
$$
Therefore \eqref{1.4} is equal to
\begin{equation*}
(-1)^k x^{-N(k-\ell)} N^{ 2\ell(k-\ell)} \frac{
e^{-(v_{k+1}+\ldots +v_{2k})}
\prod\limits_{\substack{i<j\\\epsilon_i=\epsilon_j}} (v_i-v_j)^2
\prod_{j=1}^{2k}\frac{ dv_j}{v_j^k} } { \prod\limits_{\substack{t
\leq k<q\\ \epsilon_t\neq \epsilon_q }}
 (1-x^{\epsilon_q-\epsilon_t}e^{\frac{v_q-v_t}{N}} )
 \prod\limits_{\substack{t \leq k<q\\ \epsilon_t = \epsilon_q}} (v_q-v_t)
} \;.
\end{equation*}
In the denominator, we rewrite the expression
$\prod_{\substack{t \leq
k<q\\ \epsilon_t\neq \epsilon_q }}
 (1-x^{\epsilon_q-\epsilon_t}e^{\frac{v_q-v_t}{N}} )
$
by noting that $x^{\epsilon_q-\epsilon_t}$ is $x$ if $\epsilon_q=1$,
$\epsilon_t=0$, which happens when $t=\ell+1,\dots,k$ and
$q=k+\ell+1,\dots, 2k$, and it equals $x^{-1}$ if $\epsilon_q=0$
and $\epsilon_t=1$, which happens when $t=1,\dots, \ell$ and
$q=k+1,\dots k+\ell$. Thus
\begin{equation*}
\begin{split}
\prod_{\substack{t \leq k<q\\ \epsilon_t\neq \epsilon_q }}
&
 (1-x^{\epsilon_q-\epsilon_t}e^{\frac{v_q-v_t}{N}} )
  =
 \prod_{t=\ell+1}^k\prod_{q=k+\ell+1}^{2k} (1-x
 e^{\frac{v_q-v_t}{N}} )\prod_{t=1}^\ell \prod_{q=k+1}^{k+\ell}
 (1-x^{-1}e^{\frac{v_q-v_t}{N}} )
\\
 &=
(-1)^{\ell}x^{-\ell^2} \prod_{\substack{1\leq t\leq k\\ k+1 \leq q\leq 2k\\
\epsilon_t\neq \epsilon_q }}
 (1-xe^{(\epsilon_q-\epsilon_t)\frac{v_q-v_t}{N}} )\prod_{t=1}^{l}\prod_{q=k+1}^{k+\ell}e^{\frac{v_{t}-v_{q}}{N}}
\end{split}
\end{equation*}

Multiplying by the common pre-factor of
$\frac{(-1)^{k}x^{kN}}{(k!)^2}$ gives that, up to a term of order
$1/N$ smaller,
\begin{multline}
P_{k,\ell}(x)\sim  
 (-1)^{\ell}\frac {x^{\ell (N+\ell)}N^{2\ell(k-\ell)}
}{(k!)^2} \frac 1{(2\pi i)^{2k}}\oint\dots \oint
\prod_{\substack{1\leq t\leq k\\ k+1 \leq q\leq 2k\\
\epsilon_t\neq \epsilon_q }}
 (1-xe^{(\epsilon_q-\epsilon_t)\frac{v_q-v_t}{N}} )^{-1}
\\
\prod_{t=1}^{l}\prod_{q=k+1}^{k+\ell}(e^{\frac{v_{t}-v_{q}}{N}})e^{-(v_{k+1}+\ldots +v_{2k})} \prod_{\substack{1\leq t\leq \ell,\;
k+l+1\leq q\leq 2k\\\rm or\\ \ell+1\leq t\leq k,\; k+1\leq q\leq k+\ell}} (v_q-v_t)
\prod_{\substack{1\leq i<j\leq \ell\\
\rm or\\ k+\ell+1\leq i<j\leq 2k\\
\rm or\\ \ell+1\leq i<j\leq k\\ \rm or \\ k+1\leq i<j\leq k+\ell}}
(v_j-v_i)^2 \prod_{j=1}^{2k}\frac{ dv_j}{v_j^k}
\end{multline}

We need to pick out the coefficient of $x^{cN}$ in $P_{k,\ell}(x)$.
(This coefficient is
automatically $0$ if $\ell(N+\ell)>cN$, so we need only
consider $\ell < c$.)  
We therefore need to find the coefficient of
$x^{cN-\ell(N+\ell)}=x^{(c-\ell)N-\ell^2}$ in
\begin{equation}
\prod_{\substack{1\leq t\leq k\\ k+1 \leq q\leq 2k\\
\epsilon_t\neq \epsilon_q }}
 (1-xe^{(\epsilon_q-\epsilon_t)\frac{v_q-v_t}{N}} )^{-1}.
 \end{equation}
 We can expand the above to get
\begin{equation}
\sum_{m=0}^{\infty} \sum_{\substack{b_{1}+\ldots+b_{\ell^{2}+(k-\ell)^{2}}=m \\ b_{i}\geq 0}}x^{m}\exp(\sum b_{q,t}(\epsilon_q-\epsilon_t)\frac{v_{q}-v_{t}}{N}).
\end{equation}
If we consider the pre-factor of $\prod_{t=1}^{\ell}\prod_{q=k+1}^{k+\ell}e^{\frac{v_{t}-v_{q}}{N}}$, then the required coefficient is
\begin{equation}
\label{trSym}
\tr \sym^{(c-\ell)N} \exp(\frac 1N V )
\end{equation}
where $V:=\diag(v_{q}-v_{t})$ for $q$ and $t$ such that $1\leq t\leq k,~~~ k+1 \leq q\leq 2k$ and $\epsilon_t\neq \epsilon_q.$
Next, we use Lemma~\ref{lem:geom sum}, proved below, to deduce that the expression \eqref{trSym} is
\begin{equation}
((c-\ell)N)^{k^2-2\ell(k-\ell)-1} J_\ell((c-\ell)\vec v )
\end{equation}
with
\begin{equation}
J_\ell(v_1,\dots, v_{2k})=\int_{\substack{ \sum x_{t,q}=1\\x_{tq\geq 0}}} e^{\sum
x_{tq}(\epsilon_q-\epsilon_t)(v_q-v_t) } \prod dx_{tq}
\end{equation}
where the $\ell^2+(k-\ell)^2$ variables $x_{tq}$ have indices $1\leq t\leq k$,
$k+1 \leq q\leq 2k$ with $\epsilon_t\neq \epsilon_q$, that is either $1\leq t\leq \ell$, $k+1\leq q\leq k+\ell$ or $\ell+1\leq t\leq k$, $k+\ell+1\leq q\leq 2k$.


Since $P_{k,\ell}(x)$ also has a factor of $N^{2\ell(k-\ell)}$, we
get a total contribution of $N^{k^2-1} (c-\ell)^{k^2-2\ell(k-\ell)-1} g_{k,\ell}(c-\ell)$ where
\begin{multline}
g_{k,\ell}(c-\ell)=\frac {(-1)^{\ell}}{(k!)^2} \frac 1{(2\pi
i)^{2k}}\oint\dots \oint 
J_\ell((c-\ell)\vec v )
\\
e^{-(v_{k+1}+\ldots +v_{2k})} \prod_{\substack{1\leq t\leq \ell,\;
k+\ell+1\leq q\leq 2k\\\rm or\\ \ell+1\leq t\leq k,\; k+1\leq q\leq
k+\ell}} (v_q-v_t)
\prod_{\substack{1\leq i<j\leq \ell\\
\rm or\\ k+\ell+1\leq i<j\leq 2k\\
\rm or\\ \ell+1\leq i<j\leq k\\ \rm or \\ k+1\leq i<j\leq k+\ell}}
(v_j-v_i)^2 \prod_{j=1}^{2k}\frac{ dv_j}{v_j^k}
\end{multline}

The prefactor $g_{k,\ell}(c-\ell)$ depends polynomially on $c-\ell$,
because to compute it we need to compute derivatives of $J_\ell((c-\ell) \vec
v)$ at $\vec v=0$, which are clearly polynomial in $(c-\ell)$.

Summing these over $0\leq \ell<c$ gives an expression of the form
$\gamma_k(c)N^{k^2-1}$, where
\begin{equation}
\gamma_k(c) = \sum_{0\leq \ell<c} \binom{k}{\ell}^2(c-\ell)^{k^2-2\ell(k-\ell)-1} g_{k,\ell}(c-\ell),
\end{equation}
as was to be proved.

It remains now to prove the two lemmas we have used.  This we do in the following subsections.

\subsubsection{Vanishing of an integral}
Denote by $P_{k}(x;a\epsilon_{1},\ldots,a\epsilon_{2k})$ the integral $P_{k}(x)$ over the circular contours centered in $a\epsilon_{i}$ when $\epsilon_{i}$ can be either zero or one.
\begin{lemma}\label{lem: zero terms}
Let the number of $\epsilon_{i}$ which are equal to $1$ and the number which are equal to $0$ be different. Then the integral $P_{k}(x;a\epsilon_{1},\ldots,a\epsilon_{2k})$ is identically zero.
\end{lemma}
\begin{proof}
We consider the case in which there are more zeros then ones. The case in which there are more ones then zeros, can be deduced in the same way. We can choose (without loss of generality) $\epsilon_{1},\ldots,\epsilon_{k+1}$ to be zero.
\\Denote
\begin{equation}
G(z_{1},\ldots,z_{k+1}):= e^{-N(z_{k+1}+\cdots+z_{2k})}\prod_{\substack{1\leq l\leq k\\k+1\leq q \leq 2k}}(1-e^{z_{q}-z_{l}})^{-1}\frac{\Delta(z_{1},\ldots,z_{2k})}
{\prod_{i=1}^{2k}(z_{i}-a)^{k}\prod_{i=k+2}^{2k}(z_{i})^{k}}
\end{equation}
This function is analytic around zero. The poles that arise when $z_{q}=z_{l}$ cancel with the vandermonde determinant.
  Next, we use the residue theorem in order to compute the integral. Consider the vandermonde determinant expansion:
$$ \Delta(z_{1},\ldots,z_{2k})=\sum_{\sigma\in S_{2k}}\Sgn(\sigma)(\prod_{i=1}^{2k}(z_{i})^{\sigma(i)-1})$$
By the residue theorem we need to show that the coefficient of $\prod_{i=1}^{k+1}(z_{i})^{k-1}$ in the product $G(z_{1},\ldots,z_{k+1})\Delta(z_{1},\ldots,z_{2k})$ is zero. For this purpose, since $G(z_{1},\ldots,z_{k+1})$ is analytic around zero, it is enough to show that there is no monomial term in the expansion of $\Delta(z_{1},\ldots,z_{2k})$ of the form $\prod_{i=1}^{k+2}(z_{i})^{\sigma(i)-1}$ with $\sigma(i)-1\leq k-1$ for $i=1,\ldots,k+1$. Since $\sigma$ is a permutation this is clearly the case.
\end{proof}

\subsubsection{A lemma on geometric sums}

Let $V = \diag(v_1,\dots, v_d)$ be a diagonal $d\times d$ matrix,
and $M$ a large parameter. We want to compute the asymptotic
behaviour of
\begin{equation}
\tr \sym^M \exp(\frac 1M V )  = \sum_{\substack{ k_1+\ldots +k_d=M
\\ k_1,\dots,k_d \geq 0}} \exp(\frac 1M\sum_{j=1}^d k_jv_j)
\end{equation}
This is the coefficient of $x^M$ in the power series expansion of
$$ \det(I-x\exp(\frac 1MV) )^{-1} = \frac 1{\prod_{j=1}^d (1-e^{v_j/M}x)}$$

\begin{lemma}\label{lem:geom sum}
As $M\to \infty$,
$$
\tr \sym^M \exp(\frac 1M V )=M^{d-1}
\iint_{\substack{x_1+\ldots +x_d=1\\x_j\geq 0}} e^{ \sum x_j v_j }
dx_1\dots dx_d
+O(M^{d-2})
$$
\end{lemma}
\begin{proof}
Dividing by $M^{d-1}$ we get a Riemann sum
\begin{equation*}
\frac 1{M^{d-1}} \sum_{\substack{ k_1+\ldots +k_d=M \\ k_1,\dots,k_d
\geq 0}} e^{\frac 1M\sum_{j=1}^d k_jv_j } =
\iint_{\substack{x_1+\ldots +x_d=1\\x_j\geq 0}} e^{ \sum x_j v_j }
dx_1\dots dx_d + O(\frac 1M)
\end{equation*}

\end{proof}

\subsubsection{Example: leading coefficient in the range $0<c\leq 1$ }
The leading coefficient in $I_{k}(cN,N)$ (i.e. the coefficient of $N^{k^{2}-1}$) when $0<c\leq 1$, can be obtained from \eqref{eval I first range}.
Thus for $0<c\leq 1$, we have $\gamma_k(c)=\frac{c^{k^{2}-1}}{(k^{2}-1)!}$.
We now verify that the complicated expression that we got in this section for the leading coefficient $\gamma_k(c)$, agrees with the above.
Note that because of the functional equation, Lemma \ref{functional eq}, we can conclude that this holds also in the range $(k-1)N\leq c\leq kN$

The leading coefficient of $I_{k}(cN,N)$ when $0<c\leq 1$ is $\gamma_k(c)=c^{k^{2}-1}g_{k,0}(c)$ where
\begin{multline}
g_{k,0}(c)=\frac {1}{(k!)^2} \frac 1{(2\pi
i)^{2k}}\oint\dots \oint 
J_{0}(c\vec v )
e^{-(v_{k+1}+\ldots +v_{2k})} \\
\prod_{\substack{1\leq i<j\leq k\\
\rm or\\ k+1\leq i<j\leq 2k}}
(v_j-v_i)^2 \prod_{j=1}^{2k}\frac{ dv_j}{v_j^k}
\\=\frac {1}{(k!)^2} \frac 1{(2\pi
i)^{2k}}\oint\dots \oint 
J_{0}(c\vec v )
e^{-(v_{k+1}+\ldots +v_{2k})}\\
\Delta(v_1,\ldots,v_k)^{2}\Delta(v_{k+1},\ldots,v_{2k})^{2}\prod_{j=1}^{2k}\frac{ dv_j}{v_j^k}
\end{multline}
By the residue theorem, in order to compute $g_{k,0}(c)$ we need to find the coefficient of $\prod_{j=1}^{2k}v_j^{k-1}$ in the expansion of
$$J_{0}(c\vec v )e^{-(v_{k+1}+\ldots +v_{2k})}
\Delta(v_1,\ldots,v_k)^{2}\Delta(v_{k+1},\ldots,v_{2k})^{2}.$$
Consider the vandermonde determinant expansion:
\begin{equation}
\Delta(v_1,\ldots,v_k)^{2}=\sum_{\sigma,\sigma'\in S_{k}}\Sgn(\sigma)\Sgn(\sigma')\prod_{i=1}^{k}(v_{i})^{\sigma(i)+\sigma'(i)-2}
\end{equation}
We are looking for terms of the form $\prod_{i=1}^{k}(v_{i})^{\sigma(i)+\sigma'(i)-2}$ with
$\sigma(i)+\sigma'(i)-2\leq k-1$ for all $1\leq i\leq k$. Since $\sigma(i)$ and $\sigma'(i)$ are permutations, the only such term is
$k!\prod_{i=1}^{k}(v_{i})^{k-1}$. In the same way, the only possible contribution from $\Delta(v_{k+1},\ldots,v_{2k})^{2}$ to the integral comes from the term $k!\prod_{i=k+1}^{2k}(v_{i})^{k-1}$. That means that the term $J_{0}(c\vec v )e^{-(v_{k+1}+\ldots +v_{2k})}$ can contribute only a constant.
Therefore, the calculation comes down to verifying that $J_{0}(c\vec v )=\frac{1}{(k^{2}-1)!}$ when $\vec v =0$. This is indeed the case, since when $\vec v =0$, $J_{0}(c\vec v )$ is the volume of a $k^{2}-1$ dimensional simplex, which is $\frac{1}{(k^{2}-1)!}$ as required.

\section{Justification of Conjecture~\ref{conj:mean square SI}}\label{Conj}
Our final goal is to sketch briefly a justification for
Conjecture~\ref{conj:mean square SI} without reference to the
function field results in the body of the paper.  In addition, we
indicate how to generate a conjecture for the lower order terms in
the asymptotic expansion (\ref{conj:mean square SI}), as noted at
the end of Section~\ref{Conj-intro}.

We start by defining
\begin{equation}\label{Q}
Q_k(\alpha, T)=\frac{1}{T(\log T)^{k^2}}\int_0^T\zeta(\frac{1}{2}+\frac{i\alpha}{\log T}+\i t)^k\zeta(\frac{1}{2}+\frac{i\alpha}{\log T}-\i t)^k dt.
\end{equation}
We have the Riemann-Stieljes integral identity,
$$
\zeta^k(1/2+i\alpha/\log T + it) = \int_{-\infty}^\infty e^{-i\alpha x/\log T} e^{-ixt} e^{-x/2} d\Delta_k(e^x).
$$
Substituting this into (\ref{Q}) and swapping the order of integration, we find that
$$
Q_k(\alpha, T) \sim \frac{T}{(\log T)^{k^2-1}} \int_{-\infty}^{\infty} e^{-2i\alpha u} \frac{1}{T^u} \Delta_k^2(T^u;\, T^{u-1})\,du.
$$
Hence, by Fourier inversion, on average
\begin{equation}\label{invert}
\frac{T^{1-v}}{(\log T)^{k^2-1}}\Delta_k^2(T^v;\, T^{v-1})\sim\int_{-\infty}^{\infty}Q_k(\pi\beta, T){\rm e}^{2\pi i\beta v}d\beta
\end{equation}
Conjecture~\ref{conj:mean square SI} now follows from a conjecture of K\"osters \cite{K}:
\begin{equation}
\label{zetaRMTconstant}
\lim_{T\rightarrow\infty}Q_k(\alpha, T) = a_k \lim_{N\rightarrow\infty} W_k(\alpha,N),
\end{equation}
where we write
$$
W_k(\alpha,N)= \frac{1}{N^{k^2}} \int_{U(N)} \det(1-e^{-i\alpha/N}U)^k \det(1-e^{-i\alpha/N}U^*)^k\,dU
$$
and $a_k$ is given by \eqref{Lester const}. (This is a matter of coupling equation (1.2) and Conjecture 1.2 of \cite{K}.)  We then have, using \eqref{generating} to expand the random matrix integral,
\begin{align}
\label{RMT_Riemannsum}
\notag W_k(\alpha,N) & = \frac{1}{N^{k^2}} \sum_{0\leq m \leq kN} I_k(m;N) e^{-2i\alpha m/N} \\
\notag &= \frac{1}{N^{k^2}} \sum_{0\leq m \leq kN} e^{-2i\alpha m/N}\big(\gamma_k(m/N) N^{k^2-1}+O_k(N^{k^2-2})\big) \\
\notag &\sim \frac{1}{N} \sum_{m\geq 0} e^{-2i\alpha m/N} \gamma_k(m/N) \mathbf{1}_{[0,k]}(m/N) \\
&\sim \int_0^k e^{-2i\alpha u} \gamma_k(u)\,du,
\end{align}
as the last sum is a Riemann sum.

Setting $X = T^u$ and $H=T^{u-1}$ in \eqref{invert} implies that for $x \approx X$, on average
$$
\frac{\Delta_k^2(x;\,H)}{H (\tfrac{\log X}{u})^{k^2-1}} \sim a_k \gamma_k(u).
$$
We may impose $H=X^\delta$ by setting $u = 1/(1-\delta)$. The restriction that $u\in [0,k]$ becomes $\delta\in [0,1-1/k]$ and the Conjecture follows.

The expression \eqref{zetaRMTconstant} follows from conjectures in \cite{cfkrs2} which relate $Q_k(\alpha, T)$ to a combinatorial sum, like that in Theorem~\ref{partial_sums}, and to a multiple contour integral, like that in Theorem~\ref{complex integral representation}, which include arithmetic factors \cite{K}.  Specifically, it follows from a leading-order asymptotic evaluation of the multiple contour integral that is similar to the calculation given here in Section~\ref{asymptotic2}.  A calculation of lower order terms, as in Section~\ref{sec:other ranges} of the present paper, leads to a polynomial of order $k^2-1$ in the variable $\log X$ for the second moment of $\Delta_k$.

 \end{document}